\documentclass[11pt,notitlepage,a4paper]{article}
\usepackage{graphicx}
\usepackage{amssymb}
\usepackage{amsmath}
\usepackage{appendix}

\def\tr{\mathop{\rm tr}\nolimits}
\def\Re{\mathop{\rm Re}\nolimits}

\def\diag{\mathop{\rm diag}\nolimits}
\def\tr{\mathop{\rm tr}\nolimits}

\def\etr{\mathop{\rm etr}\nolimits}

\newcommand{\0}{\boldsymbol{0}}

\newcommand{\bLambda}{\boldsymbol{\Lambda}}

\newcommand{\bSigma}{\boldsymbol{\Sigma}}
\newcommand{\bUpsilon}{\boldsymbol{\Upsilon}}

\newcommand{\bPsi}{\boldsymbol{\Psi}}
\newcommand{\bOmega}{\boldsymbol{\Omega}}
\newcommand{\ba}{\boldsymbol{a}}
\newcommand{\bB}{\boldsymbol{B}}
\newcommand{\bb}{\boldsymbol{b}}
\newcommand{\bA}{\boldsymbol{A}}

\newcommand{\bH}{\boldsymbol{H}}

\newcommand{\bI}{\boldsymbol{I}}

\newcommand{\bM}{\boldsymbol{M}}

\newcommand{\bT}{\boldsymbol{T}}

\newcommand{\bU}{\boldsymbol{U}}

\newcommand{\bV}{\boldsymbol{V}}

\newcommand{\bW}{\boldsymbol{W}}

\newcommand{\bX}{\boldsymbol{X}}

\newcommand{\bY}{\boldsymbol{Y}}

\newcommand{\bZ}{\boldsymbol{Z}}

\newcommand{\C}{\mathcal{C}}

\newcommand{\I}{\mathcal{I}}

\renewcommand{\L}{\mathcal{L}}

\renewcommand{\O}{\mathcal{O}}

\renewcommand{\S}{\mathcal{S}}

\def\tr{\mathop{\rm tr}\nolimits}
\def\Re{\mathop{\rm Re}\nolimits}

\def\diag{\mathop{\rm diag}\nolimits}
\def\tr{\mathop{\rm tr}\nolimits}

\def\etr{\mathop{\rm etr}\nolimits}

\newtheorem{theorem}{Theorem}

\newtheorem{definition}{Definition}

\newtheorem{lemma}[theorem]{Lemma}

\newtheorem{remark}{Remark}

\newenvironment{proof}[1][Proof]{\textbf{#1.} }{\ \rule{0.5em}{0.5em}}
\input{tcilatex}
\input{tcilatex}

\newcommand{\dprod}{\prod}

\begin{document}

\title{\bf Wishart Generator Distribution}

\bigskip

\author{{ A. Bekker$^2$\footnote{Corresponding Author. Email: andriette.bekker@up.ac.za}, M. Arashi$^{1,2}$ and J. van Niekerk$^2$}
\vspace{.5cm} \\\it $^1$Department of Statistics, School of
Mathematical Sciences\\\vspace{.5cm} \it University of Shahrood, Shahrood, Iran \vspace{.03cm}\\\it $^2$Department of
Statistics, Faculty of Natural and Agricultural Sciences,
\\\vspace{.5cm} \it University of Pretoria, Pretoria, 0002, South
Africa }

\date{}
\maketitle

\begin{quotation}
\noindent {\it Abstract:} The Wishart distribution and its
generalizations are among the most prominent probability
distributions in multivariate statistical analysis, arising
naturally in applied research and as a basis for theoretical
models. In this paper, we generalize the Wishart distribution
utilizing a different approach that leads to the Wishart generator
distribution with the Wishart distribution as a special case. It
is not restricted, however some special cases are exhibited.
Important statistical characteristics of the Wishart generator
distribution are derived from the matrix theory viewpoint.
Estimation is also touched upon as a guide for further research
from the classical approach as well as from the Bayesian paradigm.
The paper is concluded by giving applications of two special cases
of this distribution in calculating the product of beta functions
and astronomy.

\par

\vspace{9pt} \noindent {\it Key words and phrases:} Bayesian
estimation; Eigenvalue; Elliptically contoured distribution;
Hypergeometric function; Random matrix; Wishart distribution;
Zonal polynomial

\par

\vspace{9pt} \noindent {\it AMS Classification:} Primary: 62F15,
Secondary: 62H05\par

\end{quotation}\par

\noindent

\section{Introduction}
The Wishart distribution and its generalizations are among the
most prominent probability distributions in multivariate
statistical analysis, arising naturally in applied research and as
a basis for theoretical models. The reader is referred to Gupta
and Nagar (2000) and Anderson (2003) for a more extensive study
regarding the theoretical as well as the practical uses of the
Wishart distribution. Various generalizations and extensions are
proposed for the Wishart distribution, because of its importance
in matrix theory. To mention a few: Sutradhar and Ali (1989)
generalized the Wishart distribution for the vector variate
elliptical models, however Teng et al. (1989) considered matrix
variate elliptical models in their study. Wong and Wang (1995)
defined the Laplace-Wishart distribution, while Letac and Massam
(2001) defined the normal quasi-Wishart distribution. In the
context of graphical models, Roverato (2002) defined the
hyper-inverse Wishart and Wang and West (2009) extended the
inverse Wishart distribution for using hyper-Markov properties
(see Dawid and Lauritzen, 1993), while Bryc (2008) proposed the
compound Wishart and $q$-Wishart in graphical models. Abul-Magd
(2009) proposed a generalization to Wishart-Laguerre ensembles.
Adhikari (2010) generalized the Wishart distribution for
probabilistic structural dynamics, and D\'iaz-Garc\'ia and
Guti\'errez-J\'aimez (2011) extended the Wishart distribution for
real normed division algebras. Munilla and Cantet (2012) also
formulated a special structure for the Wishart distribution to
apply in modeling the maternal animal.

There are of course many extensions that are not listed in the
above, however the Wishart distribution can be viewed in the sense
that it gives rise to other distributions. Thus the possibility of
extending each of the previous applications of the Wishart to
hyper models, can be considered. We propose a possible
construction methodology for creating new matrix variate
distributions. The building block for our approach is discussed in
the following section. To demonstrate the novelty, we compare it
to a recent contribution by Carlo-Lopera et al. (2014) in the literature.\\

\textbf{Building Block}\\

Following Teng et al. (1989), Caro-Lopera et al. (2014) recently
proposed a generalized Wishart distribution (GWD) under the
elliptical models. They nicely derived the non-central moments of
the likelihood ratio statistic for testing the equality of two
covariance matrices under elliptical models for the corresponding
matrices. Indeed, they considered the quadratic form of a matrix
elliptical variate for building their distributions. We refer to
p. 539 of Anderson (2003) and D\'iaz-Garc\'ia and
Guti\'errez-J\'aimez (2011) for more details and extensions. To be
more specific, we recall a random matrix
$\bY\in\mathbb{R}^{n\times m}$ is said to have matrix elliptically
contoured distribution with location matrix
$\bM\in\mathbb{R}^{n\times m}$, column covariance matrix
$\bSigma\in\S_m$ and density generator
$g:\mathbb{R}^+\rightarrow\mathbb{R}^+$, denoted by $\bY\sim
EC(\bM,\bSigma,g)$, if its density function has the form
\begin{eqnarray}
f(\bY)&=&|\bSigma|^{-\frac{1}{2}}g\left[\tr(\bY-\bM)\bSigma^{-1}(\bY-\bM)^T\right]\\
&&\mbox{OR}\cr
f(\bY)&=&d_{n,m}|\bSigma|^{-\frac{1}{2}}g\left[\tr(\bY-\bM)\bSigma^{-1}(\bY-\bM)^T\right],
\end{eqnarray}
where $d_{n,m}$ is the normalizing constant.

Caro-Lopera et al. (2014) used Eq. (1) to develop generalized
Wishart, however we deem to consider Eq. (2) in our construction.
The difference in the form of the density generator $g(.)$, plays
deterministic role in extending matrix variate distributions. In
Eq. (1), the normalizing constant is included in the form of
$g(.)$, however, it is not the case for Eq. (2) and the density
generator in the latter equation can be any Borel measurable
function. Thus, considering the quadratic form $\bA=\bY^T\bY$, the
GWD based on Eq. (1) depends on the elliptical distribution,
whereas the GWD based on Eq. (2) is free of any restriction and
can take any form. The GWD based on Eq. (2) is neglected in the
literature. This family of distributions, is a rich family with
many applications. We propose some of the special members and
applications in this paper.

We organize the paper as follows: In section 2 a construction
proposition behind the Wishart generator distribution is discussed
using elementary tools in matrix theory and some of special cases
are proposed. Section 3 contains some of the important statistical
characteristic of this distribution, while a short note is given
in section 4 regarding estimation purposes. Further developments
beyond the Wishart generator distribution are given in section 5
and an application of a special case in section 5 is given in
section 6. We conclude our result in section 7 and section 8 is
devoted to some necessary tools from matrix algebra.

\section{\noindent Wishart Generator Distribution}

\noindent In this section a new family of distributions namely the \textit{%
Wishart Generator Distribution} (WGD) is defined and some special
cases along with the definition of inverse WGD are given. In a
nutshell, the new generator type distribution concludes from a
special case of Lemma \ref{lemma teng} for $\kappa=0$ (see
Appendix), which is provided in below.

\begin{definition}
\label{defintion wgd}A random matrix $\mathbf{X}\in S_{m}$ is said to have
the WGD with parameter $\mathbf{\Sigma }\in S_{m}$, degrees of freedom $%
n\geq m$ and Borel measurable function $h(\cdot )$, $h(\cdot )\neq 1$ (called shape generator), denoted by $\mathbf{%
X\thicksim }WG_{m}(\mathbf{\Sigma },n,h)$, if it has the following density function%
\begin{equation*}
f(\mathbf{X})=k_{n,m}|\mathbf{\Sigma }|^{-\frac{n}{2}}|\mathbf{X}|^{\frac{n}{%
2}-\frac{m+1}{2}}h(tr\mathbf{\Sigma }^{-1}\mathbf{X})
\end{equation*}%
where using Lemma \ref{lemma teng} for $\kappa=0$,
\begin{eqnarray*}
k_{n,m}^{-1}=\frac{\Gamma _{m}(\frac{n}{2})\gamma _{0}(\frac{n}{2})}{\Gamma (\frac{nm}{%
2})},\quad \gamma_0\left(\frac{n}{2}\right)=\int_{\mathbb{R}^+}
y^{\frac{nm}{2}-1}h(y)\textnormal{d}y
\end{eqnarray*}%
provided that the above integral exists.
\end{definition}
\begin{remark}
The shape generator in Definition \ref{defintion wgd}, should
sometimes admit the Taylor's series expansion as a regularity
condition, which will be referenced where ever needed.
\end{remark}
The reason of naming the distribution in Definition \ref{defintion
wgd} as Wishart generator, is the following result.
\begin{remark}
Setting $h(x)=exp(-\frac{x}{2})$ in Definition \ref{defintion wgd}
yields the Wishart distribution (Press, 1982, 5.1.1). Referring
back to the the building block in the Introduction, it is clear
that the form of $h(.)$ here is free of taking any normalizing
constant, however Carlo-Lopera et al. (2014) took an  specific
choice of $h(.)$ to fulfill a valid density function for the
Wishart distribution.
\end{remark}
Now we list some special cases, obtained from considering
different selections of $h$ in Definition \ref{defintion wgd}. Not
to be conservative, various combinations of hypergeometric,
trigonometric, exponential and Bessel functions can be considered
to propose a new matrix distribution followed by WGD. The only
restriction that should be fulfilled, is the existence of
$\gamma_0(.)$, i.e., $\gamma_0(\frac{n}{2})<\infty$. Looking in
this way to construct a matrix distribution is not worthwhile from
practical viewpoint, because it results in a complex structure.
However, some applications are provided in section 6 for some
special cases to address the practical importance.

\begin{enumerate}
\item[1.] Taking $h(x)=(1+x)^{-\left(\frac{nm}{2}+p\right)}$ in Definition \ref{defintion wgd}, for
$p>0$ we get the density function of a matrix variate t (MT)
distribution as
\begin{eqnarray}\label{matrix variate t}
f(\bX)=\frac{\Gamma\left(\frac{nm}{2}+p\right)}{\Gamma_m\left(\frac{n}{2}\right)\Gamma\left(p\right)}|\mathbf{\Sigma }|^{-\frac{n}{2}}|\mathbf{X}|^{\frac{n}{%
2}-\frac{m+1}{2}}\left(1+\tr\bSigma^{-1}\bX\right)^{-\left(\frac{nm}{2}+p\right)},
\end{eqnarray}
where we used
\begin{equation*}
\gamma_0\left(\frac{n}{2}\right)=\int_{\mathbb{R}^+}
y^{\frac{nm}{2}-1}(1+y)^{-\left(\frac{nm}{2}+p\right)}\textnormal{d}y=B\left(\frac{nm}{2},p\right)
\end{equation*}
\item[2.] Taking $h(x)=\exp\left(-ax^b\right)$ in Definition \ref{defintion wgd}, for
$a,b>0$, and using Eq. 3.478(1), p. 370 of Gradshteyn and Ryzhik
(2007)  we get the density function of a power Wishart
distribution as
\begin{eqnarray}\label{power Wishart}
f(\bX)=\frac{ba^{\frac{mn}{2b}}}{\Gamma\left(\frac{nm}{2b}\right)}
|\mathbf{\Sigma }|^{-\frac{n}{2}}|\mathbf{X}|^{\frac{n}{%
2}-\frac{m+1}{2}}\exp\left[-a(\tr\bSigma^{-1}\bX)^b\right].
\end{eqnarray}
\item[3.] Taking $h(x)=(a+x)^{-\frac{mn-1}{2}}\exp\left(-bx\right)$ in Definition \ref{defintion wgd}, for
$|a|<\pi$, $b>0$, and using Eq. 3.383(6), p. 348 of Gradshteyn and
Ryzhik (2007)  we get the density function of a matrix variate
Kummer-type distribution as
\begin{eqnarray}\label{matrix variate Kummer}
f(\bX)=\frac{2^{\frac{1-nm}{2}}\sqrt{b}\;|\mathbf{\Sigma
}|^{-\frac{n}{2}}}{\gamma\left(\frac{nm}{2}\right)e^{\frac{ab}{2}}D_{1-nm}\left(\sqrt{2ab}\right)}
|\mathbf{X}|^{\frac{n}{%
2}-\frac{m+1}{2}}\left(a+\tr\bSigma^{-1}\bX\right)^{-\frac{mn-1}{2}}\etr\left(-b\bSigma^{-1}\bX\right).
\end{eqnarray}
\item[4.] Taking $h(x)=\exp(-bx)\left(1-\exp(-bx)\right)^{-2}$ in Definition \ref{defintion wgd}, for
$a<1$, $b>0$, and using Eq. 3.423($3^4$), p. 358 of Gradshteyn and
Ryzhik (2007)  we get the density function of a matrix variate
logistic-type distribution as
\begin{eqnarray}\label{matrix variate logistic}
f(\bX)=\frac{ab^{\frac{nm}{2}}}{c\gamma\left(\frac{nm}{2}\right)}|\bSigma|^{-\frac{n}{2}}
|\mathbf{X}|^{\frac{n}{%
2}-\frac{m+1}{2}}\etr\left(-b\bSigma^{-1}\bX\right)\left(1-\etr\left(-b\bSigma^{-1}\bX\right)\right)^{-2},
\end{eqnarray}
where $c=\sum_{i=1}^\infty a^i i^{1-\frac{nm}{2}}$.
\item[5.] Taking $h(x)=exp\left(-ax^2\right)\sin(bx)$ in Definition \ref{defintion wgd}, for
$a>0$, and using Eq. 3.952(7), p. 503 of Gradshteyn and Ryzhik
(2007) we get the density function of a sin-Wishart distribution
as
\begin{eqnarray}\label{sin Wishart}
f(\bX)=\frac{2a^{\frac{nm+2}{4}}e^{\frac{b^2}{4a}}|\bSigma|^{-\frac{n}{2}}}
{b\Gamma\left(\frac{nm+2}{4}\right)\;_1F_1\left(1-\frac{nm}{4};\frac{3}{2};\frac{b^2}{4a}\right)}
|\mathbf{X}|^{\frac{n}{%
2}-\frac{m+1}{2}}\exp\left[-a(\tr\bSigma^{-1}\bX)^2\right]\sin\left(b\tr\bSigma^{-1}\bX\right),
\end{eqnarray}
\item[6.] Taking $h(x)=exp\left(-x\right)\ln(x)$ in Definition \ref{defintion wgd}, and using Eq. 4.352(4), p. 574 of Gradshteyn and Ryzhik
(2007) we get the density function of a logarithmic-Wishart
distribution as
\begin{eqnarray}\label{logarithmic Wishart}
f(\bX)=\frac{1}
{\Gamma'\left(\frac{nm}{2}\right)}|\bSigma|^{-\frac{n}{2}}
|\mathbf{X}|^{\frac{n}{%
2}-\frac{m+1}{2}}\etr\left(-\bSigma^{-1}\bX\right)\ln\left(\tr\bSigma^{-1}\bX\right).
\end{eqnarray}
\item[7.] Taking $h(x)=\;_pF_q(a_1,\ldots,a_p,b_1,\ldots,b_q;cx)\exp\left(-x\right)$ in Definition \ref{defintion
wgd}, for $p<q$, and using Eq. 7.522(5), p. 814 of Gradshteyn and
Ryzhik (2007) we get the density function of a hypergeometric
Wishart distribution as
\begin{eqnarray}\label{logarithmic Wishart}
f(\bX)&=&\frac{|\bSigma|^{-\frac{n}{2}}}
{\Gamma\left(\frac{nm}{2}\right)\;_{p+1}F_q\left(\frac{nm}{2},a_1,\ldots,a_p,b_1,\ldots,b_q;c\right)}\cr
&&
|\mathbf{X}|^{\frac{n}{%
2}-\frac{m+1}{2}}\;_pF_q\left(a_1,\ldots,a_p,b_1,\ldots,b_q;c\tr\bSigma^{-1}\bX\right)\etr\left(-\bSigma^{-1}\bX\right).
\end{eqnarray}
\end{enumerate}

\begin{theorem}\label{IWGD}
Let $\bX\sim WG_m(\bSigma,n,h)$. Then, $\bY=\bX^{-1}$ has an
inverted WGD, denoted as $\bY\sim IWG_m(\bSigma,n,h)$, with the
density
\begin{eqnarray*}
f(\bY)=\frac{\Gamma\left(\frac{1}{2}mn\right)}{\gamma_0\left(\frac{1}{2}n\right)\Gamma_m\left(\frac{1}{2}n\right)}\det(\bSigma)^{-\frac{1}{2}n}
\det(\bY)^{-\frac{1}{2}n-\frac{1}{2}(m+1)}
h(\tr\bSigma^{-1}\bY^{-1}).
\end{eqnarray*}
\end{theorem}
\noindent\textbf{Proof:} The result follows by the fact that under
the transformation $\bY=\bX^{-1}$, the Jacobian is given by
$J(\bX\rightarrow\bY)=\det(\bY)^{-(m+1)}$.\hfill$\blacksquare$.

The following result gives some extensions to the existing result
in the literature regarding matrix variate gamma distribution.
\begin{definition}\label{GGD}
A random matrix $\bZ\in\S_m$ is said to have the matrix variate
gamma generator distribution (GGD) with parameters
$\alpha>(m-1)/2$, $\beta>0$, $\bSigma\in\S_m$, and shape generator
$h$, denoted by $\bZ\sim GG_m(\bSigma,\alpha,\beta,h)$ if it has
the following density
\begin{eqnarray*}
f(\bZ)=\frac{\Gamma(m\alpha)}{\gamma_0(\alpha)\Gamma_m(\alpha)}\det(\bSigma)^{-\alpha}\det(\bZ)^{\alpha-\frac{1}{2}(m+1)}h(2\beta\tr\bSigma^{-1}\bZ).
\end{eqnarray*}
Further if $\bW=\bZ^{-1}$, then $\bW$ has inverted GGD with the
density
\begin{eqnarray*}
f(\bW)=\frac{\Gamma(m\alpha)}{\gamma_0(\alpha)\Gamma_m(\alpha)}\det(\bSigma)^{-\alpha}\det(\bW)^{-\alpha-\frac{1}{2}(m+1)}h(2\beta\tr\bSigma^{-1}\bW^{-1}).
\end{eqnarray*}
It is then denoted by $\bW\sim IGG_m(\bSigma,\alpha,\beta,h)$.
\end{definition}
\begin{remark}
Taking $h(x)=\exp\left(-\frac{1}{2}x\right)$ in Definition
\ref{GGD}, gives the matrix variate gamma distribution of Lukacs
and Laha (1964) and inverted matrix variate gamma of Iranmanesh et
al. (2013).
\end{remark}
Note that if we take $\alpha=n/2$ and $\beta=2$, the GGD reduces
to WGD.

\section{Properties}
Since the focus of this paper is the WGD, thus in this section we
only give some important statistical properties of the WGD. These
results can be directly derived for the IWD and GGD.

It can be directly obtained that if $\bX\sim WG_m(\bSigma,n,h)$,
then the $r$-th moment of determinant of $\bX$ is equal to
\begin{eqnarray}\label{expectation of determinant}
E\left[\det(\bX)^r\right]&=&\frac{\Gamma\left(\frac{1}{2}mn\right)\det(\bSigma)^{-\frac{1}{2}n}}
{\Gamma_m\left(\frac{1}{2}n\right)\gamma_0\left(\frac{1}{2}n\right)
}\int_{\S_m}\det(\bX)^{r+\frac{1}{2}n-\frac{1}{2}(m+1)}h(\tr(\bSigma^{-1}\bX))\textnormal{d}\bX\cr
&=&\frac{\Gamma\left(\frac{1}{2}mn\right)\det(\bSigma)^{-\frac{1}{2}n}}{\Gamma_m\left(\frac{1}{2}n\right)\gamma_0\left(\frac{1}{2}n\right)
}\frac{\Gamma_m\left(r+\frac{1}{2}n\right)\gamma_0\left(r+\frac{1}{2}n\right)\det(\bSigma)^{r+\frac{1}{2}n}}{\Gamma\left(\left(r+\frac{1}{2}n\right)m\right)}\cr
&=&\frac{\Gamma\left(\frac{1}{2}mn\right)\Gamma_m\left(r+\frac{1}{2}n\right)}{\Gamma\left(\left(r+\frac{1}{2}n\right)m\right)\Gamma_m\left(\frac{1}{2}n\right)}
\frac{\gamma_0\left(r+\frac{1}{2}n\right)}{\gamma_0\left(\frac{1}{2}n\right)}
\det(\bSigma)^r.
\end{eqnarray}
Withers and Nadarajah (2010), demonstrated that for any square
non-singular matrix $\bX$, the identity
$\log\det(\bX)=\tr\log(\bX)$ occurs. Since
$\log\det(\bX)^r=r\log\det(\bX)=r\tr\log(\bX)=\tr\log\det(\bX)^r$,
from \eqref{expectation of determinant}, for $\bX\sim
WG_m(\Sigma,n,h)$ we have
\begin{eqnarray}
E\tr\log\left[\det(\bX)^r\right] &=&
E\log\left[\det(\bX)^r\right]\cr
                    &=&\log E\left[\det(\bX)^r\right]\cr
                    &=&\log\left(\frac{\Gamma\left(\frac{1}{2}mn\right)\Gamma_m\left(r+\frac{1}{2}n\right)}{\Gamma\left(\left(r+\frac{1}{2}n\right)m\right)\Gamma_m\left(\frac{1}{2}n\right)}
\frac{\gamma_0\left(r+\frac{1}{2}n\right)}{\gamma_0\left(\frac{1}{2}n\right)}\right)+r
\log\det(\bSigma).
\end{eqnarray}
And if we put $\bSigma=\bI_m$, then it yields
\begin{eqnarray*}
E\tr\log\left[\det(\bX)^r\right] &=&
E\log\left[\det(\bX)^r\right]\cr
                    &=&\log E\left[\det(\bX)^r\right]\cr
                    &=&\log\left(\frac{\Gamma\left(\frac{1}{2}mn\right)\Gamma_m\left(r+\frac{1}{2}n\right)}{\Gamma\left(\left(r+\frac{1}{2}n\right)m\right)\Gamma_m\left(\frac{1}{2}n\right)}
\frac{\gamma_0\left(r+\frac{1}{2}n\right)}{\gamma_0\left(\frac{1}{2}n\right)}\right).
\end{eqnarray*}

Further for the expectation of zonal polynomial we have
\begin{eqnarray}\label{expectation of zonal}
E[C_\kappa(\bX)]&=&\frac{\Gamma\left(\frac{1}{2}nm\right)\det(\bSigma)^{-\frac{1}{2}n}}{\Gamma_m\left(\frac{1}{2}n\right)\gamma_0\left(\frac{1}{2}n\right)
}\int_{\S_m}\det(\bX)^{\frac{1}{2}n-\frac{1}{2}(m+1)}C_\kappa(\bX)h(\tr(\bSigma^{-1}\bX))\textnormal{d}\bX\cr
&=&\frac{\Gamma\left(\frac{1}{2}nm\right)}{\Gamma_m\left(\frac{1}{2}n\right)\gamma_0\left(\frac{1}{2}n\right)
}\;C_\kappa(\bSigma).
\end{eqnarray}

One of the important statistical characteristics of a
distribution, might be its characteristic function (c.f). In the
following result we give a closed expression for the c.f of the
WGD.

\begin{theorem}
\label{characteristic function of wgd}Suppose that $\mathbf{X\thicksim }%
WG_{m}(\mathbf{\Sigma },n,h)$ and $h$ admits Taylor's series
expansion based on zonal polynomials. The c.f is given by
\begin{equation*}
\psi_{\bX} (\mathbf{T})=\sum_{k=0}^{\infty }\sum_{\kappa
}\frac{\Gamma\left(\frac{1}{2}nm\right)\left(\frac{n}{2}\right)_\kappa}{k!\Gamma\left(\frac{1}{2}nm+k\right)}
\frac{\gamma_k\left(\frac{n}{2}\right)}{\gamma_0\left(\frac{n}{2}\right)}
C_\kappa\left(i\bT\bSigma\right).
\end{equation*}
\end{theorem}

\begin{proof}
The characteristic function is defined as
\begin{eqnarray}
\psi_{\bX} (\mathbf{T}) &=&E\left[ \etr\left( i\mathbf{TX}\right)
\right]   \notag
\\
&=&\int_{\S_{m}}k_{n,m}\etr\left( i\mathbf{TX}\right) |\mathbf{\Sigma }|^{-%
\frac{n}{2}}|\mathbf{X}|^{\frac{n}{2}-\frac{m+1}{2}}h(\tr\mathbf{\Sigma }^{-1}%
\mathbf{X})\textnormal{d}\mathbf{X}  \notag \\
&=&k_{n,m}|\mathbf{\Sigma }|^{-\frac{n}{2}}\int_{\S_{m}}\etr\left( i\mathbf{TX}%
\right) |\mathbf{X}|^{\frac{n}{2}-\frac{m+1}{2}}h(\tr\mathbf{\Sigma }^{-1}%
\mathbf{X})\textnormal{d}\mathbf{X}  \label{char func wgd 1}
\end{eqnarray}%
Using the Taylor's series expansion we have
\begin{eqnarray*}
\etr(i\bT\bX)=\sum_{k=0}^\infty
\frac{\tr(i\bT\bX)^k}{k!}=\sum_{k=0}^\infty\sum_\kappa\frac{C_\kappa(i\bT\bX)}{k!}
\end{eqnarray*}
Hence, using Lemma \ref{lemma teng} we get
\begin{eqnarray*}
\psi_{\bX} (\mathbf{T})&=&\sum_{k=0}^{\infty }\sum_{\kappa
}\frac{k_{n,m}}{k!}\det(
\mathbf{\Sigma })^{-\frac{n}{2}}\int_{\S_{m}}|\mathbf{X}|^{\frac{n}{2}-\frac{%
m+1}{2}}C_\kappa(i\bT\bX)h(\tr\bSigma^{-1}\bX)\textnormal{d}%
\mathbf{X}\cr &=& \sum_{k=0}^{\infty }\sum_{\kappa
}\frac{k_{n,m}}{k!}\det( \mathbf{\Sigma
})^{-\frac{n}{2}}\frac{\left(\frac{n}{2}\right)_\kappa\Gamma_m\left(\frac{n}{2}\right)\gamma_k\left(\frac{n}{2}\right)}{\Gamma\left(\frac{nm}{2}+k\right)}
\det(\bSigma)^{\frac{n}{2}}C_\kappa\left(i\bT\bSigma\right)\cr
&=&\sum_{k=0}^{\infty }\sum_{\kappa
}\frac{\Gamma\left(\frac{nm}{2}\right)}{k!\gamma_0\left(\frac{n}{2}\right)}\frac{\left(\frac{n}{2}\right)_\kappa\gamma_k\left(\frac{n}{2}\right)}{\Gamma\left(\frac{nm}{2}+k\right)}
C_\kappa\left(i\bT\bSigma\right).
\end{eqnarray*}%
\end{proof}

It might be ambiguous that how one can get the c.f of the Wishart
distribution using the result of Theorem \ref{characteristic
function of wgd}. Before rectifying this inconvenience, we need
the following lemma which plays a key role in deducing the c.f of
the Wishart distribution from Theorem \ref{characteristic function
of wgd}.

\begin{lemma}\label{characteristic function of Wishart}
Let $\bY\sim W_m(\bSigma,n)$ (Wishart distribution of dimension
$m$ with $n$ degrees of freedom). Then its c.f is given by
\begin{eqnarray*}
\psi_{\bY}(\bT)=\sum_{k=0}^\infty\sum_\kappa
\frac{2^{k}}{k!}\left(\frac{n}{2}\right)_\kappa
C_\kappa(i\bT\bSigma).
\end{eqnarray*}
\end{lemma}
\begin{proof}
It is easy to see that the c.f of $\bY$ has the expression
$\psi_{\bY}(\bT)=\det(\bI_m-2i\bT\bSigma)^{-\frac{n}{2}}$. However
in deriving the c.f, we make use of an integral over symmetric
positive definite matrices. This integral is equal to
\begin{eqnarray}\label{eq11}
\C\I&=&\int_{\S_m}\det(\bY)^{\frac{n}{2}-\frac{m+1}{2}}\etr\left(-\frac{1}{2}\bSigma^{-1}\bY+i\bT\bY\right)\textnormal{d}\bY\cr
&=&2^{\frac{nm}{2}}\det(\bSigma)^{\frac{n}{2}}\Gamma_m\left(\frac{n}{2}\right)\det(\bI_m-2i\bT\bSigma)^{-\frac{n}{2}},
\end{eqnarray}
On the other hand, writing the exponential term $\etr(i\bT\bY)$ in
$\C\I$ as series of zonal polynomials, by Taylor's series
expansion, and using Lemma \ref{lemma teng} for
$h(x)=\exp\left(-\frac{1}{2}x\right)$, we have that
\begin{eqnarray}\label{eq12}
\C\I&=&\sum_{k=0}^\infty\frac{1}{k!}\sum_\kappa\int_{\S_m}\det(\bY)^{\frac{n}{2}-\frac{m+1}{2}}\etr\left(-\frac{1}{2}\bSigma^{-1}\bY\right)C_\kappa(i\bT\bY)\textnormal{d}\bY\cr
&=&\sum_{k=0}^\infty\frac{1}{k!}\sum_\kappa
2^{\frac{nm}{2}+k}\left(\frac{n}{2}\right)_\kappa\Gamma_m\left(\frac{n}{2}\right)\det(\bSigma)^{\frac{n}{2}}C_\kappa(i\bT\bSigma)
\end{eqnarray}
Comparing equations \eqref{eq11} and \eqref{eq12}, yields
\begin{eqnarray}\label{eq14}
\det(\bI_m-2i\bT\bSigma)^{-\frac{n}{2}}
=\sum_{k=0}^\infty\sum_\kappa
\frac{2^{k}}{k!}\left(\frac{n}{2}\right)_\kappa
C_\kappa(i\bT\bSigma)
\end{eqnarray}
which completes the proof.
\end{proof}
\begin{remark}
Using Lemma \ref{characteristic function of Wishart}, it can be
directly followed that by taking
$h(x)=\exp\left(-\frac{1}{2}x\right)$ in Theorem
\ref{characteristic function of wgd}, we obtain the characteristic
function of Wishart distribution, since
$\gamma_k\left(\frac{n}{2}\right)=2^{\frac{nm}{2}+k}\Gamma\left(\frac{nm}{2}+k\right)$
and
$\gamma_0\left(\frac{n}{2}\right)=2^{\frac{nm}{2}}\Gamma\left(\frac{nm}{2}\right)$.
\end{remark}

\begin{remark}\label{remark laplace} If one is interested in deriving the distribution
of a trace of a matrix, it can be done through inverting the
Laplace transform. Using Theorem \ref{characteristic function of
wgd}, replacing $i$ by $i^2=-1$, it can be directly deduced that
the Laplace transform of $WG(\bSigma,n,h)$ is given by
\begin{eqnarray*}
\L(s)&=&E[\etr(-s\bX)]\cr
&=&\frac{\Gamma\left(\frac{nm}{2}\right)\det(\bSigma)^{-\frac{n}{2}}}{\gamma_0\left(\frac{n}{2}\right)}\sum_{k=0}^{\infty
}\sum_{\kappa
}\frac{s^{-\left(\frac{nm}{2}+k\right)}\left(\frac{n}{2}\right)_\kappa}{k!}
C_\kappa\left(\bSigma^{-1}\right)\cr
&=&\frac{\Gamma\left(\frac{nm}{2}\right)}{\gamma_0\left(\frac{n}{2}\right)\det(s\bSigma)^{\frac{n}{2}}}\det\left(\bI_m-(s\bSigma)^{-1}\right)\cr
&=&\frac{\Gamma\left(\frac{nm}{2}\right)}{\gamma_0\left(\frac{n}{2}\right)}\det\left(s\bSigma-\bI_m\right)\det(s\bSigma)^{-\frac{n+1}{2}}
\end{eqnarray*}
where the third equality obtained from Eq. \eqref{eq14}.
\end{remark}

\begin{theorem}
Let $\bX\sim WG_m(\bSigma,n,h)$, and $\bA\in\S_m$. Then
$\bA\bX\bA'\sim WG_m\left((\bA')^{-1}\bSigma\bA^{-1},n,h\right)$.
\end{theorem}
\noindent\textbf{Proof:} The proof follows from the fact that the
Jacobian of the transformation $\bY=\bA\bX\bA'$ is given by
$J(\bX\rightarrow\bY)=\det(\bA)^{-(m+1)}$.\hfill$\blacksquare$\\
It can be then concluded that if $\bX\sim WG_m(\bSigma,n,h)$ and
$\bSigma=\bA\bA'$, then $\bA\bX\bA'\sim WG_m(\bI_m,n,h)$.

\begin{theorem}
Let $\mathbf{X}\sim WG_{m}\left( \mathbf{\Sigma },n,h\right) $. The joint
density function of the eigenvalues $\mathbf{\Lambda }=diag(\lambda
_{1},...,\lambda _{m}),$ $\lambda _{m}>...>\lambda _{1}>0$ of $\mathbf{X}$
is given by%
\begin{eqnarray*}
g\left( \mathbf{\Lambda }\right)  &=&\frac{\pi ^{\frac{1}{2}m^{2}}\Gamma
\left( \frac{mn}{2}\right) |\mathbf{\Sigma }|^{-\frac{n}{2}}}{\Gamma
_{m}\left( \frac{m}{2}\right) \Gamma _{m}\left( \frac{n}{2}\right) \gamma
_{0}\left( \frac{n}{2}\right) }\sum_{k=1}^{\infty }\sum_{\kappa }\frac{%
h^{\left( k\right) }\left( 0\right) C_{\kappa }\left( \mathbf{\Sigma }%
^{-1}\right) }{k!C_{\kappa }\left( \mathbf{I}_{m}\right) } \\
&&\det(\bLambda)^{\frac{n}{2}-\frac{m+1}{2}}\Delta(\bLambda)C_{\kappa
}\left( \mathbf{\Lambda }\right),
\end{eqnarray*}
where $\Delta(\bLambda)$ is the repulsion factor given by
$\Delta(\bLambda)\equiv\Delta(\lambda_1,\ldots,\lambda_m)=
\dprod\limits_{i<j}^{m}\left( \lambda _{i}-\lambda _{j}\right)$.
\end{theorem}

\begin{proof}
From Theorem 3.2.17. from Muirhead (2005), the density of $\mathbf{\Lambda }$
is given by%
\begin{equation}
g\left( \mathbf{\Lambda }\right) =\frac{\pi ^{\frac{1}{2}m^{2}}}{\Gamma
_{m}\left( \frac{m}{2}\right) }\dprod\limits_{i<j}^{m}\left( \lambda
_{i}-\lambda _{j}\right) \int_{O(m)}f(\mathbf{H\Lambda H}^{\prime })d\mathbf{%
H}  \label{eigen dens 1}
\end{equation}%
Note that from Definition \ref{defintion wgd}%
\begin{equation}
\int_{O(m)}f(\mathbf{H\Lambda H}^{\prime })d\mathbf{H=}\frac{\Gamma (\frac{nm%
}{2})}{\Gamma _{m}(\frac{n}{2})\gamma _{0}(\frac{n}{2})}|\mathbf{\Sigma }|^{-%
\frac{n}{2}}|\mathbf{\Lambda }|^{\frac{n}{2}-\frac{m+1}{2}}\int_{O(m)}h(tr%
\mathbf{\Sigma }^{-1}\mathbf{H\Lambda H}^{\prime })d\mathbf{H}
\label{igen dens 2}
\end{equation}%
Since $h(\cdot )$ admits the Taylor expansion,%
\begin{eqnarray}
\int_{O(m)}h(tr\mathbf{\Sigma }^{-1}\mathbf{H\Lambda H}^{\prime })d\mathbf{H}
&\mathbf{=}&\sum_{k=1}^{\infty }\sum_{\kappa }\frac{h^{\left( k\right)
}\left( 0\right) }{k!}\int_{O(m)}C_{\kappa }\left( \mathbf{\Sigma }^{-1}%
\mathbf{H\Lambda H}^{\prime }\right) d\mathbf{H}  \notag \\
&\mathbf{=}&\sum_{k=1}^{\infty }\sum_{\kappa }\frac{h^{\left( k\right)
}\left( 0\right) C_{\kappa }\left( \mathbf{\Sigma }^{-1}\right) C_{\kappa
}\left( \mathbf{\Lambda }\right) }{k!C_{\kappa }\left( \mathbf{I}_{m}\right)
}  \label{eigen dens 3}
\end{eqnarray}%
From (\ref{eigen dens 1}), (\ref{igen dens 2})\ and (\ref{eigen dens 3})%
\begin{eqnarray*}
g\left( \mathbf{\Lambda }\right)  &=&\frac{\pi ^{\frac{1}{2}m^{2}}\Gamma (%
\frac{nm}{2})}{\Gamma _{m}\left( \frac{m}{2}\right) \Gamma _{m}(\frac{n}{2}%
)\gamma _{0}(\frac{n}{2})}\dprod\limits_{i<j}^{m}\left( \lambda _{i}-\lambda
_{j}\right) |\mathbf{\Sigma }|^{-\frac{n}{2}} \\
&&\times |\mathbf{\Lambda }|^{\frac{n}{2}-\frac{m+1}{2}}\sum_{k=1}^{\infty
}\sum_{\kappa }\frac{h^{\left( k\right) }\left( 0\right) C_{\kappa }\left(
\mathbf{\Sigma }^{-1}\right) C_{\kappa }\left( \mathbf{\Lambda }\right) }{%
k!C_{\kappa }\left( \mathbf{I}_{m}\right) }
\end{eqnarray*}%
Note that $|\mathbf{\Lambda }|^{\frac{n}{2}-\frac{m+1}{2}}=\dprod%
\limits_{i=1}^{m}\lambda _{i}^{\frac{n}{2}-\frac{m+1}{2}}$, hence%
\begin{eqnarray*}
g\left( \mathbf{\Lambda }\right)  &=&\frac{\pi ^{\frac{1}{2}m^{2}}\Gamma (%
\frac{nm}{2})}{\Gamma _{m}\left( \frac{m}{2}\right) \Gamma _{m}(\frac{n}{2}%
)\gamma _{0}(\frac{n}{2})}\dprod\limits_{i<j}^{m}\left( \lambda _{i}-\lambda
_{j}\right) |\mathbf{\Sigma }|^{-\frac{n}{2}} \\
&&\times \dprod\limits_{i=1}^{m}\lambda _{i}^{\frac{n}{2}-\frac{m+1}{2}%
}\sum_{k=1}^{\infty }\sum_{\kappa }\frac{h^{\left( k\right) }\left( 0\right)
C_{\kappa }\left( \mathbf{\Sigma }^{-1}\right) C_{\kappa }\left( \mathbf{%
\Lambda }\right) }{k!C_{\kappa }\left( \mathbf{I}_{m}\right) }
\end{eqnarray*}
\end{proof}

\begin{theorem}
Let $\mathbf{X}\sim WG_{m}\left( \mathbf{\Sigma },n,h\right) $. Then for any
$A\in S_{m}$%
\begin{equation*}
P(\mathbf{X}<\mathbf{A})=\frac{\Gamma (\frac{nm}{2})}{\Gamma _{m}(\frac{n}{2}%
)\gamma _{0}(\frac{n}{2})}|\mathbf{\Sigma }|^{-\frac{n}{2}%
}\sum_{k=1}^{\infty }\sum_{\kappa }\frac{h^{\left( k\right) }\left( 0\right)
\Gamma _{m}\left( \frac{n}{2},\kappa \right) \Gamma _{m}\left( \frac{m+1}{2}%
,\kappa \right) }{k!\Gamma _{m}\left( \frac{n}{2}+\frac{m+1}{2},\kappa
\right) }C_{\kappa }\left( \mathbf{\Sigma }^{-1}\mathbf{A}^{\frac{1}{2}}%
\mathbf{A}^{\frac{1}{2}}\right)
\end{equation*}
\end{theorem}

\begin{proof}
Note that%
\begin{equation*}
P(\mathbf{X}<\mathbf{A})=\frac{\Gamma (\frac{nm}{2})}{\Gamma _{m}(\frac{n}{2}%
)\gamma _{0}(\frac{n}{2})}|\mathbf{\Sigma }|^{-\frac{n}{2}}\int_{\mathbf{0}<%
\mathbf{X}<\mathbf{A}}|\mathbf{X}|^{\frac{n}{2}-\frac{m+1}{2}}h(tr\mathbf{%
\Sigma }^{-1}\mathbf{X})d\mathbf{X}
\end{equation*}%
Now make the transformation $\mathbf{Y}=\mathbf{A}^{-\frac{1}{2}}\mathbf{XA}%
^{-\frac{1}{2}}$, then the Jacobian is $J(\mathbf{X}\rightarrow \mathbf{Y})=|%
\mathbf{A}|^{m+1}$ hence%
\begin{eqnarray*}
P(\mathbf{X} &<&\mathbf{A})=\frac{\Gamma (\frac{nm}{2})}{\Gamma _{m}(\frac{n%
}{2})\gamma _{0}(\frac{n}{2})}|\mathbf{\Sigma }|^{-\frac{n}{2}}\int_{\mathbf{%
0}<\mathbf{Y}<\mathbf{I}_{m}}|\mathbf{Y}|^{\frac{n}{2}-\frac{m+1}{2}}h(tr%
\mathbf{\Sigma }^{-1}\mathbf{A}^{\frac{1}{2}}\mathbf{YA}^{\frac{1}{2}})d%
\mathbf{Y} \\
\text{ \ \ \ \ \ \ } &=&\frac{\Gamma (\frac{nm}{2})}{\Gamma _{m}(\frac{n}{2}%
)\gamma _{0}(\frac{n}{2})}|\mathbf{\Sigma }|^{-\frac{n}{2}%
}\sum_{k=1}^{\infty }\sum_{\kappa }\frac{h^{\left( k\right) }\left( 0\right)
}{k!}\int_{\mathbf{0}<\mathbf{Y}<\mathbf{I}_{m}}|\mathbf{Y}|^{\frac{n}{2}-%
\frac{m+1}{2}}C_{\kappa }\left( \mathbf{\Sigma }^{-1}\mathbf{A}^{\frac{1}{2}}%
\mathbf{YA}^{\frac{1}{2}}\right) d\mathbf{Y} \\
&=&\frac{\Gamma (\frac{nm}{2})}{\Gamma _{m}(\frac{n}{2})\gamma _{0}(\frac{n}{%
2})}|\mathbf{\Sigma }|^{-\frac{n}{2}}\sum_{k=1}^{\infty }\sum_{\kappa }\frac{%
h^{\left( k\right) }\left( 0\right) \Gamma _{m}\left( \frac{n}{2},\kappa
\right) \Gamma _{m}\left( \frac{m+1}{2},\kappa \right) }{k!\Gamma _{m}\left(
\frac{n}{2}+\frac{m+1}{2},\kappa \right) }C_{\kappa }\left( \mathbf{\Sigma }%
^{-1}\mathbf{A}^{\frac{1}{2}}\mathbf{A}^{\frac{1}{2}}\right)
\end{eqnarray*}
\end{proof}

\begin{remark}
Note that $\lambda _{\left( m\right) }<a$ is equivalent to $\mathbf{X}<a%
\mathbf{I}_{m}$ since $\mathbf{H\Lambda H}^{\prime }\mathbf{=X}$. To obtain
the cumulative distribution function of $\lambda _{\left( m\right) },$ the
largest eigenvalue of $\mathbf{X}$ the previous theorem can therefore be
used with $\mathbf{A}=a\mathbf{I}_{m}.$
\end{remark}

\begin{theorem}
Let $\mathbf{X}\sim WG_{m}\left( \mathbf{\Sigma },n,h\right) $. The
cumulative distribution function of $\lambda _{\left( m\right) },$ the
largest eigenvalue of $\mathbf{X}$ is%
\begin{equation*}
F_{\lambda _{\left( m\right) }}(a)=\frac{\Gamma (\frac{nm}{2})}{\Gamma _{m}(%
\frac{n}{2})\gamma _{0}(\frac{n}{2})}|\mathbf{\Sigma }|^{-\frac{n}{2}%
}\sum_{k=1}^{\infty }\sum_{\kappa }\frac{h^{\left( k\right) }\left( 0\right)
a^{k}\Gamma _{m}\left( \frac{n}{2},\kappa \right) \Gamma _{m}\left( \frac{m+1%
}{2},\kappa \right) }{k!\Gamma _{m}\left( \frac{n}{2}+\frac{m+1}{2},\kappa
\right) }C_{\kappa }\left( \mathbf{\Sigma }^{-1}\right)
\end{equation*}
\end{theorem}

\begin{proof}
From Theorem 6 the cumulative distribution function of $\lambda
_{\left(
m\right) },$ the largest eigenvalue of $\mathbf{X}$ is%
\begin{eqnarray*}
F_{\lambda _{\left( m\right) }}(a) &=&P(\mathbf{X}<a\mathbf{I}_{m}) \\
&=&\frac{\Gamma (\frac{nm}{2})}{\Gamma _{m}(\frac{n}{2})\gamma _{0}(\frac{n}{%
2})}|\mathbf{\Sigma }|^{-\frac{n}{2}}\sum_{k=1}^{\infty }\sum_{\kappa }\frac{%
h^{\left( k\right) }\left( 0\right) a^{k}\Gamma _{m}\left( \frac{n}{2}%
,\kappa \right) \Gamma _{m}\left( \frac{m+1}{2},\kappa \right) }{k!\Gamma
_{m}\left( \frac{n}{2}+\frac{m+1}{2},\kappa \right) }C_{\kappa }\left(
\mathbf{\Sigma }^{-1}a^{\frac{1}{2}}a^{\frac{1}{2}}\right)  \\
&=&\frac{\Gamma (\frac{nm}{2})}{\Gamma _{m}(\frac{n}{2})\gamma _{0}(\frac{n}{%
2})}|\mathbf{\Sigma }|^{-\frac{n}{2}}\sum_{k=1}^{\infty }\sum_{\kappa }\frac{%
h^{\left( k\right) }\left( 0\right) a^{k}\Gamma _{m}\left( \frac{n}{2}%
,\kappa \right) \Gamma _{m}\left( \frac{m+1}{2},\kappa \right) }{k!\Gamma
_{m}\left( \frac{n}{2}+\frac{m+1}{2},\kappa \right) }C_{\kappa }\left(
\mathbf{\Sigma }^{-1}\right)
\end{eqnarray*}%
since $C_{\kappa }\left( \mathbf{\Sigma }^{-1}a^{\frac{1}{2}}a^{\frac{1}{2}%
}\right) =C_{\kappa }\left( a\mathbf{\Sigma }^{-1}\right) =a^{k}C_{\kappa
}\left( \mathbf{\Sigma }^{-1}\right) .$
\end{proof}

\begin{theorem}\label{theorem distribution of trace} Let $\bX\sim WG_m(\bSigma,n,h)$. Then $y=\tr(\bX)$
has the following density function
\begin{equation*}
f(y)=\frac{\Gamma\left(\frac{nm}{2}\right)\det(\bSigma)^{-\frac{n}{2}}\exp(-y)}{\gamma_0\left(\frac{n}{2}\right)}\sum_{k=0}^{\infty
}\sum_{\kappa
}\frac{\left(\frac{n}{2}\right)_\kappa}{k!\Gamma\left(\frac{nm}{2}+k\right)}
C_\kappa\left(\bSigma^{-1}\right)y^{\frac{nm}{2}+k-1}
\end{equation*}
\end{theorem}
\begin{proof} By applying inverse Laplace transform, using Remark \ref{remark
laplace}, we get
\begin{eqnarray}
f(y)&=&\frac{1}{2\pi
i}\lim_{T\rightarrow\infty}\int_{c-iT}^{c+iT}\exp\left[sy\right]\L(s)\textnormal{d}s\cr
           &=&\frac{\Gamma\left(\frac{nm}{2}\right)\det(\bSigma)^{-\frac{n}{2}}}{\gamma_0\left(\frac{n}{2}\right)}\sum_{k=0}^{\infty
}\sum_{\kappa }\frac{\left(\frac{n}{2}\right)_\kappa}{k!}
C_\kappa\left(\bSigma^{-1}\right)\frac{1}{2\pi
i}\lim_{T\rightarrow\infty}\int_{c-iT}^{c+iT}s^{-\left(\frac{nm}{2}+k\right)}\exp\left[sy\right]\textnormal{d}s\cr
&=&\frac{\Gamma\left(\frac{nm}{2}\right)\det(\bSigma)^{-\frac{n}{2}}}{\gamma_0\left(\frac{n}{2}\right)}\sum_{k=0}^{\infty
}\sum_{\kappa }\frac{\left(\frac{n}{2}\right)_\kappa}{k!}
C_\kappa\left(\bSigma^{-1}\right)\frac{1}{\Gamma\left(\frac{nm}{2}+k\right)}y^{\frac{nm}{2}+k-1}e^{-y}
\end{eqnarray}
\end{proof}

From Theorem \ref{theorem distribution of trace}, it can be easily
concluded that for $\bX\sim WG_m(\bSigma,n,h)$, the $r$-th moment
of $\tr(\bX)$ has the form
\begin{eqnarray}\label{expectation of tr}
E\left[(\tr\bX)^r\right]=\frac{\Gamma\left(\frac{nm}{2}\right)\det(\bSigma)^{-\frac{n}{2}}}{\gamma_0\left(\frac{n}{2}\right)}\sum_{k=0}^{\infty
}\sum_{\kappa
}\frac{\left(\frac{n}{2}\right)_\kappa\Gamma\left(\frac{nm}{2}+k+r\right)}{k!\Gamma\left(\frac{nm}{2}+k\right)}
C_\kappa\left(\bSigma^{-1}\right).
\end{eqnarray}
In the following result, we derive the distribution of the ratios
of the WGD in connection with the Wishart distribution.
\begin{theorem}\label{ratio}
Let $\bX\sim WG_m(\alpha\bSigma,n,h)$ be independent of $\bY\sim
W_m(\beta\bSigma,p)$. Then
\begin{enumerate}
\item[(i)] The r.v.
$\bB_1=\bX^{-\frac{1}{2}}\bY\bX^{-\frac{1}{2}}$ has the following
density function
\begin{eqnarray*}
g_1(\bB_1)&=&\frac{k_{n,m}\Gamma_m\left(\frac{n+p}{2}\right)}{2^{\frac{pm}{2}}\Gamma_m\left(\frac{p}{2}\right)}
\left(\frac{\alpha}{\beta}\right)^{\frac{pm}{2}}\cr
&&\det(\bB_1)^{\frac{p}{2}-\frac{m+1}{2}}
\sum_{k=0}^\infty\sum_{\kappa}\frac{1}{k!}\left(-\frac{\alpha}{2\beta}\right)^k\frac{\left(\frac{n+p}{2}\right)_\kappa
\gamma_\kappa\left(\frac{n+p}{2}\right)}{\Gamma\left(m(\frac{n+p}{2})+k\right)}C_\kappa(\bB_1).
\end{eqnarray*}
\item[(ii)] The r.v. $\bB_2=(\bX+\bY)^{-\frac{1}{2}}\bX(\bX+\bY)^{-\frac{1}{2}}$ has the following
density function
\begin{eqnarray*}
g_2(\bB_2)&=&\frac{k_{n,m}\Gamma_m\left(\frac{n+p}{2}\right)}{2^{\frac{pm}{2}}\Gamma_m\left(\frac{p}{2}\right)}
\left(\frac{\alpha}{\beta}\right)^{\frac{pm}{2}} \cr
&&\det(\bB_2)^{-\frac{n}{2}-\frac{m+1}{2}}\det(\bI_m-\bB_2)^{\frac{p}{2}-\frac{m+1}{2}}\cr
          &&\sum_{k=0}^\infty\sum_{\kappa}\frac{1}{k!}\left(-\frac{\alpha}{2\beta}\right)^k
          \frac{\left(\frac{n+p}{2}\right)_\kappa
\gamma_\kappa\left(\frac{n+p}{2}\right)}{\Gamma\left(m(\frac{n+p}{2})+k\right)}C_\kappa(\bB_2^{-1}-\bI_m).
\end{eqnarray*}
\end{enumerate}
\end{theorem}
\noindent\textbf{Proof}: The joint density function of $(\bX,\bY)$
is given by
\begin{eqnarray*}
f(\bX,\bY)=C\det(\bX)^{\frac{n}{2}-\frac{m+1}{2}}\det(\bY)^{\frac{p}{2}-\frac{m+1}{2}}
\etr\left(-\frac{\beta^{-1}}{2}\bSigma^{-1}\bY\right)h(\alpha^{-1}\tr\bSigma^{-1}\bX),
\end{eqnarray*}
where
\begin{eqnarray*}
C=\frac{k_{n,m}}{2^{\frac{pm}{2}}\Gamma_m\left(\frac{p}{2}\right)}
\alpha^{-\frac{nm}{2}}\beta^{-\frac{pm}{2}}\det(\bSigma)^{-\frac{n+p}{2}}.
\end{eqnarray*}
Make the transformations
$\bB_1=\bX^{-\frac{1}{2}}\bY\bX^{-\frac{1}{2}}$ and $\bU=\bX$ with
the Jacobian
$J(\bX,\bY\rightarrow\bB_1,\bU)=\det(\bU)^{\frac{1}{2}(m+1)}$ to
get
\begin{eqnarray*}
g(\bB_1,\bU)&=&f(\bU,\bU^{\frac{1}{2}}\bB_1\bU^{\frac{1}{2}})\cr
&=&C\det(\bU)^{\frac{n+p}{2}-\frac{m+1}{2}}\det(\bB_1)^{\frac{p}{2}-\frac{m+1}{2}}\etr\left(-\frac{1}{2\beta}\bSigma^{-1}\bU^{\frac{1}{2}}\bB_1\bU^{\frac{1}{2}}\right)
h(\alpha^{-1}\bSigma^{-1}\bU).
\end{eqnarray*}
For a moment, assume that the distribution of $\bB_1$ is
symmetric. Thus from symmetrized density we have
\begin{eqnarray*}
g_1(\bB_1)&=&\int_{\O(m)}g_1(\bH\bB_1\bH')\textnormal{d}\bH\cr
&=&\int_{\O(m)}\int_{\S_m}g(\bH\bB_1\bH',\bU)\textnormal{d}\bU\textnormal{d}\bH\cr
&=&C\det(\bB_1)^{\frac{p}{2}-\frac{m+1}{2}}\int_{\S_m}\det(\bU)^{\frac{n+p}{2}-\frac{m+1}{2}}
h(\alpha^{-1}\bSigma^{-1}\bU)\cr
&&\left(\int_{\O(m)}\etr\left(-\frac{1}{2\beta}\bU^{\frac{1}{2}}\bSigma^{-1}\bU^{\frac{1}{2}}\bH\bB_1\bH^T\right)\textnormal{d}\bH\right)\textnormal{d}\bU\cr
&=&C\det(\bB_1)^{\frac{p}{2}-\frac{m+1}{2}}\sum_{k=0}^\infty\sum_{\kappa}\frac{1}{k!}\left(-\frac{1}{2\beta}\right)^k\frac{C_\kappa(\bB_1)}{C_\kappa(\bI_m)}\cr
&&\int_{\S_m}\det(\bU)^{\frac{n+p}{2}-\frac{m+1}{2}}
C_\kappa(\bSigma^{-1}\bU)h(\alpha^{-1}\bSigma^{-1}\bU)\textnormal{d}\bU.
\end{eqnarray*}
By making use of Lemma \ref{lemma teng}, we get
\begin{eqnarray*}
g_1(\bB_1)&=&C\det(\bB_1)^{\frac{p}{2}-\frac{m+1}{2}}\sum_{k=0}^\infty\sum_{\kappa}\frac{1}{k!}\left(-\frac{1}{2\beta}\right)^k\frac{C_\kappa(\bB_1)}{C_\kappa(\bI_m)}\cr
&&\frac{\left(\frac{n+p}{2}\right)_\kappa\Gamma_m\left(\frac{n+p}{2}\right)\gamma_\kappa\left(\frac{n+p}{2}\right)\alpha^{m\left(\frac{n+p}{2}\right)+k}}
{\Gamma\left(m(\frac{n+p}{2})+k\right)}\det(\bSigma)^{\frac{n+p}{2}}C_\kappa(\bI_m).
\end{eqnarray*}
After simplification, we obtain (i). For (ii), make the
transformations
$\bB_2=(\bX+\bY)^{-\frac{1}{2}}\bX(\bX+\bY)^{-\frac{1}{2}}$ and
$\bV=\bX+\bY$, with the Jacobian
$J(\bX,\bY\rightarrow\bB_2,\bV)=\det(\bV)^{\frac{m+1}{2}}$ to get
\begin{eqnarray*}
g_2(\bB_2)&=&C\det(\bB_2)^{\frac{p}{2}-\frac{m+1}{2}}\det(\bI_m-\bB_2)^{\frac{p}{2}-\frac{m+1}{2}}\cr
          &&\int_{\S_m}\det(\bV)^{\frac{n+p}{2}-\frac{m+1}{2}}\etr\left(-\frac{1}{2\beta}\bV^{\frac{1}{2}}\bSigma^{-1}\bV^{\frac{1}{2}}[\bI_m-\bB_2]\right)
          h(\alpha^{-1}\tr\bV^{\frac{1}{2}}\bSigma^{-1}\bV^{\frac{1}{2}}\bB_2)\textnormal{d}\bV.
\end{eqnarray*}
Make the transformation
$\bZ=\bV^{\frac{1}{2}}\bSigma^{-1}\bV^{\frac{1}{2}}$ with the
Jacobian $J(\bV\rightarrow\bZ)=\det(\bSigma)^{-\frac{m+1}{2}}$ and
use Lemma \ref{lemma teng} to obtain
\begin{eqnarray*}
g_2(\bB_2)&=&C\det(\bB_2)^{\frac{p}{2}-\frac{m+1}{2}}\det(\bI_m-\bB_2)^{\frac{p}{2}-\frac{m+1}{2}}\det(\bSigma)^{\frac{n+p}{2}}\cr
          &&\int_{\S_m}\det(\bV)^{\frac{n+p}{2}-\frac{m+1}{2}}\etr\left(-\frac{1}{2\beta}\bZ[\bI_m-\bB_2]\right)
          h(\alpha^{-1}\tr\bZ\bB_2)\textnormal{d}\bZ\cr
          &=&C\det(\bB_2)^{\frac{p}{2}-\frac{m+1}{2}}\det(\bI_m-\bB_2)^{\frac{p}{2}-\frac{m+1}{2}}\det(\bSigma)^{\frac{n+p}{2}}\cr
          &&\sum_{k=0}^\infty\sum_{\kappa}\frac{1}{k!}\left(-\frac{1}{2\beta}\right)^k
          \int_{\S_m}\det(\bV)^{\frac{n+p}{2}-\frac{m+1}{2}}C_\kappa\left(\bZ[\bI_m-\bB_2]\right)
          h(\alpha^{-1}\tr\bZ\bB_2)\textnormal{d}\bZ\cr
          &=&C\det(\bB_2)^{\frac{p}{2}-\frac{m+1}{2}}\det(\bI_m-\bB_2)^{\frac{p}{2}-\frac{m+1}{2}}\det(\bSigma)^{\frac{n+p}{2}}\cr
          &&\sum_{k=0}^\infty\sum_{\kappa}\frac{1}{k!}\left(-\frac{1}{2\beta}\right)^k
          \frac{\left(\frac{n+p}{2}\right)_\kappa\Gamma_m\left(\frac{n+p}{2}\right)\gamma_\kappa\left(\frac{n+p}{2}\right)\alpha^{m(\frac{n+p}{2})+k}}
          {\Gamma\left(m(\frac{n+p}{2})+k\right)}\det(\bB_2)^{-\frac{n+p}{2}}C_\kappa(\bB_2^{-1}-\bI_m)
\end{eqnarray*}
After simplification, gives (ii) and the proof is
complete.\hfill$\blacksquare$
\begin{remark} One way of checking the accuracy of the result of
Theorem \ref{ratio}, is to consider whether one can get the same
result by taking $h(x)=\exp\left(-\frac{1}{2}x\right)$ for the
Wishart distribution. It is well established that if $\bX\sim
W_m(\bSigma,n)$, then $\bB_1$ has the well-known beta type II
distribution. This result directly follows by making use of Eq.
\eqref{eq14}. It can be also shown that $\bB_2$ has the beta type
I distribution if we take $h(.)$ to be of exponential form.
\end{remark}

\section{Estimation}
In this section, we briefly consider some estimation aspects for
the WG distribution, including the classical as well as Bayesian
viewpoints. The focus is the latter paradigm.
\subsection{Maximum likelihood estimation}
In this section, we derive a non-linear equation to find the
maximum likelihood estimator (MLE) of $\bSigma$ along with Fisher
information matrix.
\begin{theorem}\label{MLE}
Let $\mathbf{X}\sim WG_{m}\left( \mathbf{\Sigma },n,h\right) $,
where the trio $(m,n,h)$ is assumed to be known. Further assume
that $h\left( \cdot \right) $ is a monotonic continuous and
differentiable function. Then the MLE of $\mathbf{\Sigma }$ is
given by
\begin{equation*}
\hat{\bSigma} =
\frac{2}{n}g'\left(\tr(\hat{\bSigma}^{-1}\bX)\right)\cdot\bX,
\end{equation*}%
where $g(.)=-\log[h(.)]$ and $g'(x)=dg(x)/dx$.
\end{theorem}
\begin{proof}
The likelihood function is given by
\begin{equation*}
L\left( \mathbf{\Sigma }\right) \varpropto
\det(\bSigma)^{-\frac{n}{2}}\det(\bX)^{\frac{n}{2}-\frac{m+1}{2}}h(tr\mathbf{\Sigma
}^{-1}\mathbf{X})
\end{equation*}%
Hence the log-likelihood function is%
\begin{eqnarray*}
l\left( \mathbf{\Sigma }\right)  &\varpropto &\frac{n}{2}\log
\det(\bSigma^{-1}) +\left( \frac{n}{2}-\frac{m+1}{2}\right) \log
\det(\bX) +\log \left[ h(tr\mathbf{\Sigma }^{-1}\mathbf{X})\right]
\\
&\varpropto &\frac{n}{2}\log \det(\bSigma^{-1}) +\log \left[
h(tr\mathbf{\Sigma }^{-1}\mathbf{X})\right]
\end{eqnarray*}%
To obtain the maximum of the log-likelihood function, let
$\bZ=tr\mathbf{\Sigma }^{-1}\mathbf{X}$; then differentiated
log-likelihood function has the form
\begin{eqnarray*}
\frac{\partial l\left( \mathbf{\Sigma }\right)}{\partial
\mathbf{\Sigma^{-1} }}
&\varpropto &\frac{n}{2}\left[2\bSigma-\diag(\bSigma)\right]-%
\frac{d g(\bZ)}{d \bZ}[2\bX-\diag\bX].
\end{eqnarray*}%
Setting $\frac{\partial l\left( \mathbf{\Sigma }\right)}{\partial
\mathbf{\Sigma^{-1} }}$ to zero gives the MLE of $\bSigma$ as
\begin{eqnarray*}
\hat{\bSigma} &=& \frac{2}{n}g'(\bZ)\cdot\bX\cr
              &=&
              \frac{2}{n}g'\left(\tr(\hat{\bSigma}^{-1}\bX)\right)\cdot\bX.
\end{eqnarray*}
\end{proof}

Since the structure discussed in Theorem \ref{MLE} is similar to
the generalized elliptical distributions studied by Frahm (2004),
we do not provide inferential aspects of the MLE here and for
complete explanations on the MLE regarding existence, consistency,
applications and etc., the reader is referred to Frahm (2004).

\subsection{Bayesian estimation}

\begin{theorem}
Let $\mathbf{X|\Sigma }\sim WG_{m}\left( \mathbf{\Sigma },n,h\right) $.
Suppose that the prior distribution of  $\mathbf{\Sigma }$ is an inverse
Wishart distribution with parameters $\mathbf{\Omega }$ and $p$, hence $%
\mathbf{\Sigma }\sim W_{m}^{-1}\left( \mathbf{\Omega },p\right) .$ The
marginal distribution of $\mathbf{X}$ is given by%
\begin{eqnarray*}
m(\mathbf{X}) &=&\frac{\Gamma (\frac{nm}{2})}{2^{\frac{p\left( p-m-1\right)
}{2}}\Gamma _{m}(\frac{p}{2})\Gamma _{m}(\frac{n}{2})\gamma _{0}(\frac{n}{2})%
}\det(\bOmega)^{\frac{p-m-1}{2}} \\
&&\times \det(\bX)^{-\frac{p}{2}-\frac{m+1}{2}}\sum_{k=1}^{\infty
}\sum_{\kappa }\frac{\left( \frac{n+p}{2}\right) _{\kappa }\Gamma
_{m}\left(
\frac{n+p}{2}\right) \gamma _{k}\left( \frac{n+p}{2}\right) }{\Gamma \left( m%
\frac{n+p}{2}+k\right) }C_{\kappa }\left( -\frac{1}{2}\mathbf{\Omega X}%
^{-1}\right)
\end{eqnarray*}
\end{theorem}
\begin{proof}
The marginal distribution of $\mathbf{X}$ is given by%
\begin{eqnarray*}
m(\mathbf{X}) &=&\int_{S_{m}}f\left( \mathbf{X|\Sigma }\right) \pi \left(
\mathbf{\Sigma }\right) \textnormal{d}\mathbf{\Sigma } \\
&=&\frac{\Gamma (\frac{nm}{2})}{2^{\frac{p\left( p-m-1\right) }{2}}\Gamma
_{m}(\frac{p}{2})\Gamma _{m}(\frac{n}{2})\gamma _{0}(\frac{n}{2})}\det(\bX)^{\frac{n}{2}-\frac{m+1}{2}}\det(\bOmega)^{\frac{p-m-1}{2}} \\
&&\times \int_{S_{m}}\det(\bSigma)^{-\frac{n}{2}-\frac{p}{2}}\etr\left( -%
\frac{1}{2}\mathbf{\Sigma }^{-1}\mathbf{\Omega }\right) h(\tr\mathbf{\Sigma }%
^{-1}\mathbf{X})\textnormal{d}\mathbf{\Sigma }.
\end{eqnarray*}%
Now, let $\mathbf{\Sigma }^{-1}=\mathbf{T}$ then the Jacobian is $\det(\bT)^{-m-1}$ hence%
\begin{eqnarray*}
m(\mathbf{X}) &=&\frac{\Gamma (\frac{nm}{2})}{2^{\frac{p\left( p-m-1\right)
}{2}}\Gamma _{m}(\frac{p}{2})\Gamma _{m}(\frac{n}{2})\gamma _{0}(\frac{n}{2})%
}\det(\bX)^{\frac{n}{2}-\frac{m+1}{2}}\det(\bOmega)^{\frac{p-m-1}{2}}
\\
&&\times \int_{S_{m}}\det(\bT)^{\frac{n+p-m-1}{2}}\etr\left( -\frac{1}{2}%
\mathbf{T\Omega }\right) h(\tr\mathbf{TX})\textnormal{d}\mathbf{T} \\
&=&\frac{\Gamma (\frac{nm}{2})}{2^{\frac{p\left( p-m-1\right) }{2}}\Gamma
_{m}(\frac{p}{2})\Gamma _{m}(\frac{n}{2})\gamma _{0}(\frac{n}{2})}\det(\bX)^{\frac{n}{2}-\frac{m+1}{2}}\det(\bOmega)^{\frac{p-m-1}{2}} \\
&&\times \sum_{k=1}^{\infty }\sum_{\kappa }\int_{S_{m}}\det(\bT)^{\frac{%
n+p-m-1}{2}}C_{\kappa }\left( -\frac{1}{2}\mathbf{T\Omega }\right) h(\tr%
\mathbf{TX})\textnormal{d}\mathbf{T}.
\end{eqnarray*}%
Using Lemma \ref{lemma teng} we get
\begin{eqnarray*}
&&\int_{S_{m}}\det(\bT)^{\frac{n+p-m-1}{2}}C_{\kappa }\left( -\frac{1}{2}%
\mathbf{T\Omega }\right) h(\tr\mathbf{TX})\textnormal{d}\mathbf{T} \\
&=&\frac{\left( \frac{n+p}{2}\right) _{\kappa }\Gamma _{m}\left( \frac{n+p}{2%
}\right) \gamma _{k}\left( \frac{n+p}{2}\right) }{\Gamma \left( m\frac{n+p}{2%
}+k\right) }\det(\bX)^{-\frac{n+p}{2}}C_{\kappa }\left( -\frac{1}{2}%
\mathbf{\Omega X}^{-1}\right)
\end{eqnarray*}%
Hence%
\begin{eqnarray*}
m(\mathbf{X}) &=&\frac{\Gamma (\frac{nm}{2})}{2^{\frac{p\left( p-m-1\right)
}{2}}\Gamma _{m}(\frac{p}{2})\Gamma _{m}(\frac{n}{2})\gamma _{0}(\frac{n}{2})%
}\det(\bOmega)^{\frac{p-m-1}{2}} \\
&&\times \det(\bX)^{-\frac{p}{2}-\frac{m+1}{2}}\sum_{k=1}^{\infty
}\sum_{\kappa }\frac{\left( \frac{n+p}{2}\right) _{\kappa }\Gamma
_{m}\left(
\frac{n+p}{2}\right) \gamma _{k}\left( \frac{n+p}{2}\right) }{\Gamma \left( m%
\frac{n+p}{2}+k\right) }C_{\kappa }\left( -\frac{1}{2}\mathbf{\Omega X}%
^{-1}\right)
\end{eqnarray*}
\end{proof}

\begin{theorem}\label{posterior distribution }
Let $\mathbf{X|\Sigma }\sim WG_{m}\left( \mathbf{\Sigma
},n,h\right) $. Suppose that the prior distribution of
$\mathbf{\Sigma }$ is an inverse
Wishart distribution with parameters $\mathbf{\Omega }$ and $p$, hence $%
\mathbf{\Sigma }\sim W_{m}^{-1}\left( \mathbf{\Omega },p\right) .$ The
posterior distribution of $\mathbf{\Sigma }$ is given by%
\begin{eqnarray*}
\pi \left( \mathbf{\Sigma |X}\right)  &=&\det(\bX)^{\frac{n+p}{2}}\det(\bSigma)^{-\frac{n}{2}-\frac{p}{2}}\etr\left( -\frac{1}{2}\mathbf{%
\Sigma }^{-1}\mathbf{\Omega }\right) h(\tr\mathbf{\Sigma }^{-1}\mathbf{X}) \\
&&\times \left[ \sum_{k=1}^{\infty }\sum_{\kappa }\frac{\left( \frac{n+p}{2}%
\right) _{\kappa }\Gamma _{m}\left( \frac{n+p}{2}\right) \gamma _{k}\left(
\frac{n+p}{2}\right) }{\Gamma \left( m\frac{n+p}{2}+k\right) }C_{\kappa
}\left( -\frac{1}{2}\mathbf{\Omega X}^{-1}\right) \right] ^{-1}
\end{eqnarray*}
\end{theorem}

\begin{proof}
The posterior distribution is from Bayes' theorem as
\begin{equation*}
\pi \left( \mathbf{\Sigma |X}\right) =\frac{f\left( \mathbf{X|\Sigma }%
\right) \pi \left( \mathbf{\Sigma }\right) }{m(\mathbf{X})}
\end{equation*}%
Hence%
\begin{eqnarray*}
\pi \left( \mathbf{\Sigma |X}\right)  &=&\frac{\Gamma (\frac{nm}{2})}{2^{%
\frac{p\left( p-m-1\right) }{2}}\Gamma _{m}(\frac{p}{2})\Gamma _{m}(\frac{n}{%
2})\gamma _{0}(\frac{n}{2})}\det(\bX)^{\frac{n}{2}-\frac{m+1}{2}}\det(\bOmega)^{\frac{p-m-1}{2}} \\
&&\times \det(\bSigma)^{-\frac{n}{2}-\frac{p}{2}}\etr\left( -\frac{1}{2}%
\mathbf{\Sigma }^{-1}\mathbf{\Omega }\right) h(\tr\mathbf{\Sigma }^{-1}%
\mathbf{X}) \\
&&\times \frac{2^{\frac{p\left( p-m-1\right) }{2}}\Gamma _{m}(\frac{p}{2}%
)\Gamma _{m}(\frac{n}{2})\gamma _{0}(\frac{n}{2})}{\Gamma
(\frac{nm}{2})}\det(\bOmega)^{-\frac{p-m-1}{2}}\det(\bX)^{\frac{p}{2}+\frac{m+1}{2}}
\\
&&\times \left[ \sum_{k=1}^{\infty }\sum_{\kappa }\frac{\left( \frac{n+p}{2}%
\right) _{\kappa }\Gamma _{m}\left( \frac{n+p}{2}\right) \gamma _{k}\left(
\frac{n+p}{2}\right) }{\Gamma \left( m\frac{n+p}{2}+k\right) }C_{\kappa
}\left( -\frac{1}{2}\mathbf{\Omega X}^{-1}\right) \right] ^{-1} \\
&=&\det(\bX)^{\frac{n+p}{2}}\det(\bSigma)^{-\frac{n}{2}-\frac{p}{2}%
}\etr\left( -\frac{1}{2}\mathbf{\Sigma }^{-1}\mathbf{\Omega }\right) h(tr%
\mathbf{\Sigma }^{-1}\mathbf{X}) \\
&&\times \left[ \sum_{k=1}^{\infty }\sum_{\kappa }\frac{\left( \frac{n+p}{2}%
\right) _{\kappa }\Gamma _{m}\left( \frac{n+p}{2}\right) \gamma _{k}\left(
\frac{n+p}{2}\right) }{\Gamma \left( m\frac{n+p}{2}+k\right) }C_{\kappa
}\left( -\frac{1}{2}\mathbf{\Omega X}^{-1}\right) \right] ^{-1}
\end{eqnarray*}
\end{proof}

\begin{theorem}
Let $\mathbf{X|\Sigma }\sim WG_{m}\left( \mathbf{\Sigma },n,h\right) $.
Suppose that the prior distribution of $\mathbf{\Sigma }$ is an inverse
Wishart distribution with parameters $\mathbf{\Omega }$ and $p$, hence $%
\mathbf{\Sigma }\sim W_{m}^{-1}\left( \mathbf{\Omega },p\right) $. Then the
Bayes estimator of $|\mathbf{\Sigma }|$\ under the squared error loss
function is%
\begin{equation*}
\frac{\sum_{l=1}^{\infty }\sum_{\lambda }\frac{\left( \frac{n+p}{2}-1\right)
_{\lambda }\Gamma _{m}\left( \frac{n+p}{2}-1\right) \gamma _{l}\left( \frac{%
n+p}{2}-1\right) }{\Gamma \left( m\frac{n+p}{2}-m+l\right) }C_{\lambda
}\left( -\frac{1}{2}\mathbf{\Omega X}^{-1}\right) }{\sum_{k=1}^{\infty
}\sum_{\kappa }\frac{\left( \frac{n+p}{2}\right) _{\kappa }\Gamma _{m}\left(
\frac{n+p}{2}\right) \gamma _{k}\left( \frac{n+p}{2}\right) }{\Gamma \left( m%
\frac{n+p}{2}+k\right) }C_{\kappa }\left( -\frac{1}{2}\mathbf{\Omega X}%
^{-1}\right) }\det(\bX)
\end{equation*}
\end{theorem}

\begin{proof}
The Bayes estimator of $\det(\bSigma)$\ under the squared error
loss
function is%
\begin{eqnarray*}
\widehat{\det(\bSigma)} &=&E\left[ \det(\bSigma)\mathbf{|X}\right]
\\
&=&\int_{S_{m}}\det(\bX)^{\frac{n+p}{2}}\det(\bSigma)^{-\frac{n}{2}-%
\frac{p}{2}+1}\etr\left( -\frac{1}{2}\mathbf{\Sigma }^{-1}\mathbf{\Omega }%
\right) h(\tr\mathbf{\Sigma }^{-1}\mathbf{X}) \\
&&\times \left[ \sum_{k=1}^{\infty }\sum_{\kappa }\frac{\left( \frac{n+p}{2}%
\right) _{\kappa }\Gamma _{m}\left( \frac{n+p}{2}\right) \gamma _{k}\left(
\frac{n+p}{2}\right) }{\Gamma \left( m\frac{n+p}{2}+k\right) }C_{\kappa
}\left( -\frac{1}{2}\mathbf{\Omega X}^{-1}\right) \right] ^{-1}\textnormal{d}\mathbf{%
\Sigma } \\
&=&\det(\bX)^{\frac{n+p}{2}}\left[ \sum_{k=1}^{\infty }\sum_{\kappa }%
\frac{\left( \frac{n+p}{2}\right) _{\kappa }\Gamma _{m}\left( \frac{n+p}{2}%
\right) \gamma _{k}\left( \frac{n+p}{2}\right) }{\Gamma \left( m\frac{n+p}{2}%
+k\right) }C_{\kappa }\left( -\frac{1}{2}\mathbf{\Omega X}^{-1}\right) %
\right] ^{-1} \\
&&\times \sum_{l=1}^{\infty }\sum_{\lambda }\int_{S_{m}}\det(\bT)^{\frac{%
n+p-2-m-1}{2}}C_{\lambda }\left( -\frac{1}{2}\mathbf{T\Omega }\right) h(\tr%
\mathbf{TX})\textnormal{d}\mathbf{T}
\end{eqnarray*}%
where $\mathbf{T}=\mathbf{\Sigma }^{-1}$. Note that%
\begin{eqnarray*}
&&\int_{S_{m}}\det(\bT)^{\frac{n+p-2-m-1}{2}}C_{\lambda }\left( -\frac{1}{%
2}\mathbf{T\Omega }\right) h(\tr\mathbf{TX})\textnormal{d}\mathbf{T} \\
&\mathbf{=}&\frac{\left( \frac{n+p}{2}-1\right) _{\lambda }\Gamma _{m}\left(
\frac{n+p}{2}-1\right) \gamma _{l}\left( \frac{n+p}{2}-1\right) }{\Gamma
\left( m\frac{n+p}{2}-m+l\right) }|X|^{-\frac{n+p}{2}+1}C_{\lambda }\left( -%
\frac{1}{2}\mathbf{\Omega X}^{-1}\right)
\end{eqnarray*}%
from Lemma \ref{lemma teng}. Hence
\begin{eqnarray*}
\widehat{\det(\bSigma)} &=&\det(\bX)\left[ \sum_{k=1}^{\infty
}\sum_{\kappa }\frac{\left( \frac{n+p}{2}\right) _{\kappa }\Gamma
_{m}\left(
\frac{n+p}{2}\right) \gamma _{k}\left( \frac{n+p}{2}\right) }{\Gamma \left( m%
\frac{n+p}{2}+k\right) }C_{\kappa }\left( -\frac{1}{2}\mathbf{\Omega X}%
^{-1}\right) \right] ^{-1} \\
&&\times \sum_{l=1}^{\infty }\sum_{\lambda }\frac{\left( \frac{n+p}{2}%
-1\right) _{\kappa }\Gamma _{m}\left( \frac{n+p}{2}-1\right) \gamma
_{k}\left( \frac{n+p}{2}-1\right) }{\Gamma \left( m\frac{n+p}{2}-m+k\right) }%
C_{\lambda }\left( -\frac{1}{2}\mathbf{\Omega X}^{-1}\right)  \\
&=&\frac{\sum_{l=1}^{\infty }\sum_{\lambda }\frac{\left( \frac{n+p}{2}%
-1\right) _{\kappa }\Gamma _{m}\left( \frac{n+p}{2}-1\right) \gamma
_{k}\left( \frac{n+p}{2}-1\right) }{\Gamma \left( m\frac{n+p}{2}-m+k\right) }%
C_{\lambda }\left( -\frac{1}{2}\mathbf{\Omega X}^{-1}\right) }{%
\sum_{k=1}^{\infty }\sum_{\kappa }\frac{\left( \frac{n+p}{2}\right) _{\kappa
}\Gamma _{m}\left( \frac{n+p}{2}\right) \gamma _{k}\left( \frac{n+p}{2}%
\right) }{\Gamma \left( m\frac{n+p}{2}+k\right) }C_{\kappa }\left( -\frac{1}{%
2}\mathbf{\Omega X}^{-1}\right) }\det(\bX)
\end{eqnarray*}
\end{proof}

\section{\noindent Further Developments}
In this section we provide the reader with some plausible
extensions of WG distribution. In this respect, we first define
the hypergeometric WGD as in below.
\begin{definition}
\label{defintion hwgd}A random matrix $\mathbf{X}\in S_{m}$ is
said to have the hypergeometric WGD with parameters
$a_1,\ldots,a_p\in\mathbb{C}$, $b_1,\ldots,b_q\in\mathbb{C}$,
($p\leq q$), $\bOmega,\mathbf{\Sigma }\in S_{m}$, degrees of
freedom $n\geq m$ and shape generator $h(\cdot ),h(\cdot )\neq 1$, if it has the following density function%
\begin{equation*}
f(\mathbf{X})=l_{n,m}\det(\bSigma)^{-\frac{n}{2}}\det(\bX)^{\frac{n}{%
2}-\frac{m+1}{2}} \ _pF_q
\left(a_1,\ldots,a_p;b_1,\ldots,b_q;\bOmega\bX\right)h(\tr\mathbf{\Sigma
}^{-1}\mathbf{X})
\end{equation*}%
where
\begin{eqnarray*}
l_{n,m}^{-1} &=&\det(\bSigma)^{-\frac{n}{2}}\int_{S_{m}}\det(\bX)^{%
\frac{n}{2}-\frac{m+1}{2}}\;_pF_q
\left(a_1,\ldots,a_p;b_1,\ldots,b_q;\bOmega\bX\right)h(\tr\mathbf{\Sigma }^{-1}\mathbf{X})\textnormal{d}\mathbf{X} \\
&=&\det(\bSigma)^{-\frac{n}{2}}\sum_{k=0}^\infty\sum_\kappa
\frac{(a_1)_\kappa,\ldots,(a_p)_\kappa}{(b_1)_\kappa,\ldots,(b_q)_\kappa}
\frac{1}{k!}\cr &&
\int_{S_{m}}\det(\bX)^{%
\frac{n}{2}-\frac{m+1}{2}}h(\tr\mathbf{\Sigma
}^{-1}\mathbf{X})C_\kappa(\bOmega\bX)\textnormal{d}\mathbf{X}\cr
&=&
\det(\bSigma)^{-\frac{n}{2}}\sum_{k=0}^\infty\sum_\kappa\frac{(a_1)_\kappa,\ldots,(a_p)_\kappa}{(b_1)_\kappa,\ldots,(b_q)_\kappa}
\frac{1}{k!}\frac{\left(\frac{n}{2}\right)_\kappa\Gamma
_{m}(\frac{n}{2})\gamma
_{k}(\frac{n}{2})}{\Gamma (\frac{nm}{2}+k)}\det(\bSigma)^{%
\frac{n}{2}} C_\kappa(\bOmega\bSigma)\cr &=& \Gamma
_{m}\left(\frac{n}{2}\right)\sum_{k=0}^\infty\sum_\kappa\frac{(a_1)_\kappa,\ldots,(a_p)_\kappa}{(b_1)_\kappa,\ldots,(b_q)_\kappa}
\frac{\left(\frac{n}{2}\right)_\kappa\gamma
_{k}(\frac{n}{2})}{k!\Gamma
(\frac{nm}{2}+k)}C_\kappa(\bOmega\bSigma)
\end{eqnarray*}%
from Lemma \ref{lemma teng}. We designate this by $\bX\sim
HWG_{m}(\mathbf{\Sigma},\bOmega,\ba,\bb,n,h)$, where
$\ba=(a_1,\ldots,a_p)$ and $\bb=(b_1,\ldots,b_q)$.
\end{definition}
As a direct consequence of Definition \ref{defintion hwgd}, taking
$p=0$, $q=1$, $b_1=\frac{n}{2}$,
$\bOmega=\frac{1}{4}\bPsi\bSigma^{-1}$, for $\bPsi\in S_m$, gives
the non-central WGD as in below.
\begin{definition}
\label{defintion nwgd}A random matrix $\mathbf{X}\in S_{m}$ is
said to have
the non-central WGD with parameters $\bPsi,\mathbf{\Sigma }\in S_{m}$, degrees of freedom $%
n\geq m$ and shape generator $h(\cdot ),h(\cdot )\neq 1$, denoted by $\bX\sim NWG_{m}(\mathbf{\Sigma},\bOmega,n,h)$, if it has the following density function%
\begin{equation*}
f(\mathbf{X})=l_{n,m}\det(\bSigma)^{-\frac{n}{2}}\det(\bX)^{\frac{n}{%
2}-\frac{m+1}{2}} \ _0F_1
\left(\frac{1}{2}n;\frac{1}{4}\bPsi\bSigma^{-1}\bX\right)h(\tr\mathbf{\Sigma
}^{-1}\mathbf{X})
\end{equation*}%
where the normalizing constant is given by
\begin{eqnarray*}
l_{n,m}^{-1} &=&\Gamma
_{m}\left(\frac{n}{2}\right)\sum_{k=0}^\infty\sum_\kappa
\left(\frac{1}{4}\right)^{k}\frac{\gamma
_{k}(\frac{n}{2})}{k!\Gamma (\frac{nm}{2}+k)}C_\kappa(\bPsi),
\end{eqnarray*}%
since
$C_\kappa\left(\frac{1}{4}\bPsi\right)=\left(\frac{1}{4}\right)^{k}C_\kappa(\bPsi)$.
\end{definition}
Another interesting distribution raises from Definition
\ref{defintion hwgd}, comes up by setting $p$ and $q$ to 0 and 1,
respectively as:
\begin{equation}
f(\mathbf{X})=l_{n,m}\det(\bSigma)^{-\frac{n}{2}}\det(\bX)^{\frac{n}{%
2}-\frac{m+1}{2}} \etr(\bOmega\bX)h(\tr\mathbf{\Sigma
}^{-1}\mathbf{X}).
\end{equation}
We call this distribution as the exponentiated WG distribution.
Note that according to Theorem \ref{posterior distribution }, the
posterior distribution of $\bSigma$ has the exponentiated WG
distribution.

\section{\noindent Applications}
In this section, we briefly consider some applications of two
special cases of WGD.

Distributions of the form \eqref{matrix variate t} has many
applications. Arashi et al. (2013) showed that the posterior
distribution of scale matrix in the matrix variate t-population
under Jeffreys' prior has the MT distribution given by
\eqref{matrix variate t}. Another interesting application of the
MT distribution is the following result, where we show that finite
product of beta functions can be written as a ratio of gamma
functions.
\begin{theorem} Let $p>m$, then
\begin{equation*}
\prod_{i=0}^{n+1}B\left(\frac{m}{2},p+(i-2)\frac{m}{2}\right)=\frac{\Gamma_m\left(\frac{n}{2}\right)\Gamma\left(p\right)}{\Gamma\left(\frac{nm}{2}+p\right)}.
\end{equation*}
\end{theorem}
\begin{proof}
Using Corollary 3.2.3 of Srivastava and Khatri (1979) for
$\bZ=(\bZ_1,\ldots,\bZ_{n})$, $Z_i\in\mathbb{R}^{m\times 1}$ we
have
\begin{eqnarray*}
\I &=&\int_{\S_m}\det(\bX)^{\frac{n}{%
2}-\frac{m+1}{2}}\left(1+\tr\bX\right)^{-\left(\frac{nm}{2}+p\right)}\textnormal{d}\bX\cr
   &=&\int_{\mathbb{R}^n}\left(1+\sum_{i=1}^{n}\bZ_i^T\bZ_i\right)^{-\left(\frac{nm}{2}+p\right)}\textnormal{d}\bZ_1\ldots\textnormal{d}\bZ_{n}\cr
   &\overset{u_i=\bZ_i^T\bZ_i}{=}&\int_{(0,1)^n}\prod_{i=1}^n u_i^{\frac{m}{2}-1}\left(1+\sum_{i=1}^{n}u_i\right)^{-\left(\frac{nm}{2}+p\right)}\textnormal{d}u_1\ldots\textnormal{d}u_n\cr
   &=&\int_{(0,1)^n}u_1^{\frac{m}{2}-1}\prod_{i=1}^n
   u_i^{\frac{m}{2}-1}(1+u_1)^{-\left(\frac{nm}{2}+p\right)}\left(1+\sum_{i=1}^{n}\frac{u_i}{1+u_1}\right)^{-\left(\frac{nm}{2}+p\right)}\textnormal{d}u_1\ldots\textnormal{d}u_n.
\end{eqnarray*}
Now apply the transformation $v_i=\frac{u_i}{1+u_1}$, for
$i=2,\ldots,n$, with the Jacobian $J(u_2,\ldots,u_n\rightarrow
v_2,\ldots,v_n)=(1+u_1)^{n-1}$ to obtain
\begin{eqnarray*}
u_2=v_2(1+u_1),\quad\prod_{i=2}^{n}
u_i^{\frac{p}{2}-1}=(1+u_1)^{(n-1)\left(\frac{p}{2}-1\right)}\prod_{i=2}^n
v_i^{\frac{p}{2}-1}.
\end{eqnarray*}
Hence we get
\begin{eqnarray*}
\I &=&
\int_{(0,1)}u_1^{\frac{m}{2}-1}(1+u_1)^{-\left(\frac{nm}{2}+p\right)+\frac{(n+1)m}{2}}\textnormal{d}u_1\cr
   &&\times\int_{(0,1)^{n-1}}\prod_{i=2}^n
   v_i^{\frac{m}{2}-1}\left(1+\sum_{i=2}^{n}v_i\right)^{-\left(\frac{nm}{2}+p\right)}\textnormal{d}v_2\ldots\textnormal{d}v_n\cr
   &=&B\left(\frac{m}{2},p-m\right)\cr
   &&\int_{(0,1)^{n-1}}u_2^{\frac{m}{2}-1}\prod_{i=3}^n
   v_i^{\frac{m}{2}-1}(1+v_2)^{-\left(\frac{nm}{2}+p\right)}\left(1+\sum_{i=3}^{n}\frac{v_i}{1+v_1}\right)^{-\left(\frac{nm}{2}+p\right)}\textnormal{d}v_2\ldots\textnormal{d}v_n.
\end{eqnarray*}
Again make the transformation $w_i=\frac{v_i}{1+v_2}$, for
$i=3,\ldots,n$, with the Jacobian $J(v_3,\ldots,v_n\rightarrow
w_3,\ldots,w_n)=(1+v_2)^{n-2}$ to get
\begin{eqnarray*}
\I
&=&B\left(\frac{m}{2},p-m\right)B\left(\frac{m}{2},p-m+\frac{m}{2}\right)\cr
&&\times\int_{(0,1)^{n-2}}\prod_{i=3}^n
   w_i^{\frac{m}{2}-1}\left(1+\sum_{i=3}^{n}w_i\right)^{-\left(\frac{nm}{2}+p\right)}\textnormal{d}w_3\ldots\textnormal{d}w_n\cr.
\end{eqnarray*}
Continuing this procedure, finally yields
\begin{eqnarray*}
\I &=&
\prod_{i=0}^{n+1}B\left(\frac{m}{2},p+(i-2)\frac{m}{2}\right).
\end{eqnarray*}
But since $\int_{S_m}f(\bX)\textnormal{d}\bX=1$, from Eq.
\eqref{matrix variate t} for $\bSigma=\bI_m$, we have
\begin{eqnarray*}
\int_{\S_m}\det(\bX)^{\frac{n}{%
2}-\frac{m+1}{2}}\left(1+\tr\bX\right)^{-\left(\frac{nm}{2}+p\right)}\textnormal{d}\bX=
\frac{\Gamma_m\left(\frac{n}{2}\right)\Gamma\left(p\right)}{\Gamma\left(\frac{nm}{2}+p\right)},
\end{eqnarray*}
which by substituting in $\I$, completes the proof.
\end{proof}

For considering another application, let $\bY\sim
EC(\bM,\bSigma,g)$ and consider the distribution of
$\bZ=\bY^T\bY$. It is well-known that if $\bY$ has matrix variate
normal distribution, then $\bZ$ has Wishart distribution. For a
moment let
$\bUpsilon=\bSigma^{-\frac{1}{2}}\bM^T\bM\bSigma^{-\frac{1}{2}}$.
Anderson and Fang (1982) derived the density of $\bZ$ for the case
$\bM=\0$ and $\bUpsilon=\bI_m$. Fan (1984) extended their result
by presenting the density of $\bZ$ for general $\bM$ and
$\bUpsilon$ as an integral form. Afterward, Teng et al. (1989)
derived the closed form of the density $\bZ$ for practical use.

Now as an application, we show that the distribution of $\bZ$ is
the non-central WG. To see this, consider that using Theorem 1 of
Teng et al. (1989), if $\bY\sim EC(\bM,\bSigma,g)$ then the
distribution of $\bZ=\bY^T\bY$ is given by
\begin{eqnarray}
f(\bZ)=\frac{\pi^{\frac{mn}{2}}}{\Gamma_m\left(\frac{n}{2}\right)}|\bSigma|^{-\frac{n}{2}}|\bZ|^{\frac{n}{2}-\frac{m+1}{2}}
\sum_{k=0}^\infty\frac{g^{(2k)}\left(\tr(\bSigma^{-1}\bZ+\bUpsilon)\right)}{k!}
\sum_\kappa\frac{C_\kappa\left(\bUpsilon\bSigma^{-\frac{1}{2}}\bZ\bSigma^{-\frac{1}{2}}\right)}{\left(\frac{n}{2}\right)_k},~~~~
\end{eqnarray}
where $g^{(2k)}(.)$ is the $2k$-th derivative of $g(.)$.

If we take $p=0$, $q=1$, $b_1=\left(\frac{n}{2}\right)$,
$\bOmega=\bSigma^{-\frac{1}{2}}\bUpsilon\bSigma^{-\frac{1}{2}}$
and $h(x)=g^{(2k)}(x+\tr\bUpsilon)$, then using Definition
\ref{defintion nwgd}, $\bZ\sim NWG_m(\bSigma,\bOmega,n,h)$.

For an application of the $NWG_m(\bSigma,\bOmega,n,h)$ was
introduced here, consider the use of the non-central WG
distribution when it arises from lighter/heavier marginal tail
alternatives to the matrix variate Gaussian distribution, in
astronomy (see Feigelson and Babu, 2012 for applications of
statistics in astronomy). To be more precise, in the study of
imaging extrasolar planets for life, as discussed by Tourneret et
al. (2005), direct imaging through statistical signal processing
is the only method for exoplanet detection. See Figure 1 (adopted
from Google Images) to set the platform for investigating the true
distribution in the forthcoming explanation.
\begin{figure}\label{extraplanet}
\begin{tabular}{cc}
\includegraphics[scale=.7]{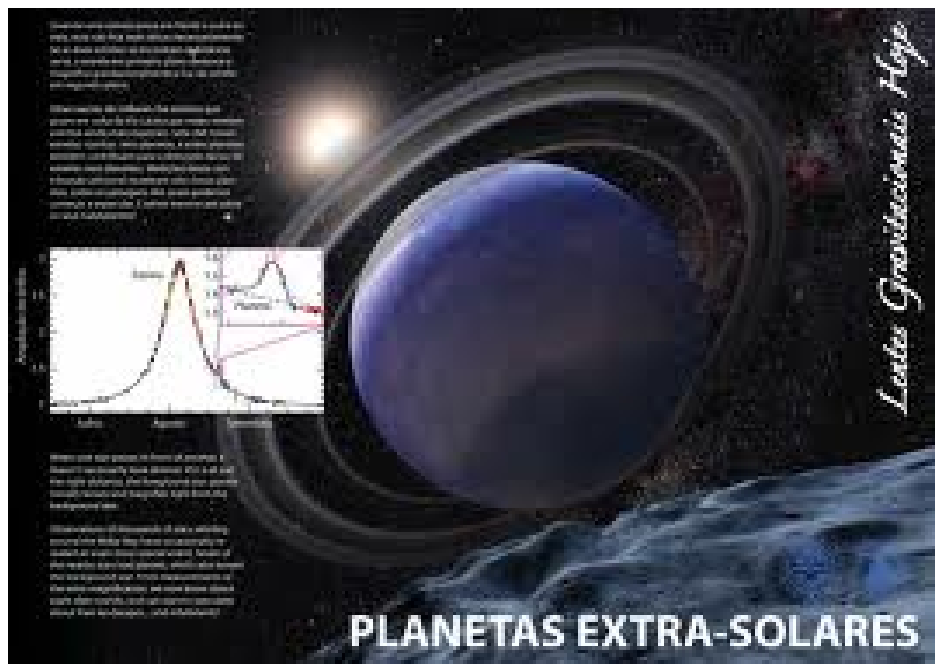}
&
\includegraphics[scale=.7]{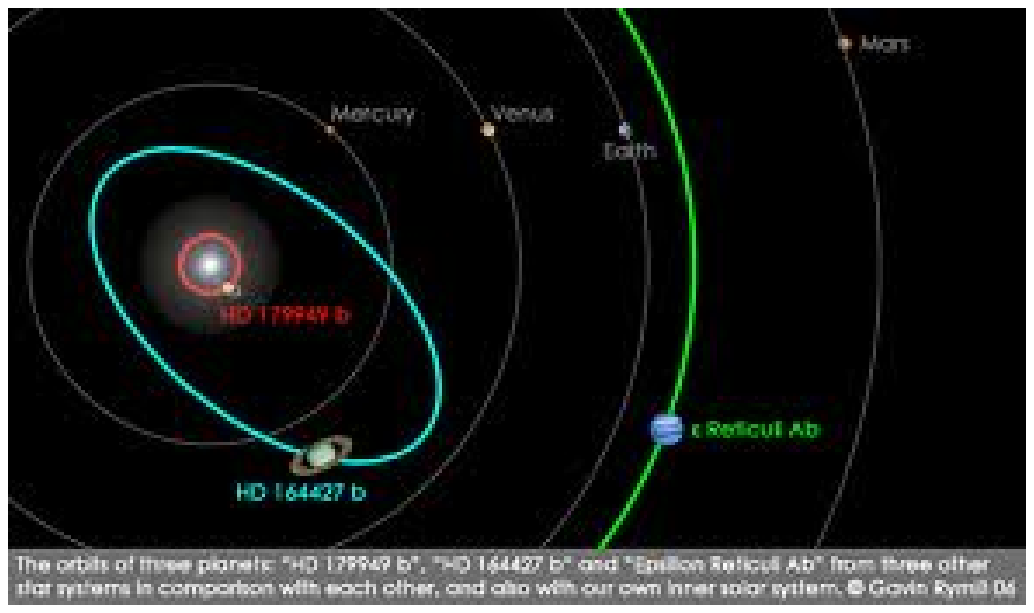}
\end{tabular}
\caption{Visualizing extra-solar planets from their position
distribution, as it might be seen from telescope.}
\end{figure}

As explained by Aime and Soummer (2004) the complex amplitude of a
wave in the focal plane of a telescope, at a position $(x,y)$, can
be written as follows:
\begin{equation*}
\psi(x,y)=C(x,y)+S(x,y),
\end{equation*}
where $C(x,y)\in\mathbb{C}$ is a deterministic term proportional
to the wave amplitude in absence of turbulence and
$S(x,y)\in\mathbb{C}$ is the wavefront amplitude (associated to
the speckles) distributed according to a zero mean complex
Gaussian distribution. Tourneret et al. (2005) assumed that the
telescope aperture has central symmetries which imply
$C(x,y)\in\mathbb{R}$ and using the fact that the real and
imaginary parts of $\psi(x,y)$, denoted by $\psi_r(x,y)$ and
$\psi_i(x,y)$ have Gaussian distributions, extended the
instantaneous intensity of the wave in the focal plane at a
position $(x,y)$, given by
\begin{eqnarray*}
\Lambda(x,y)=|\psi(x,y)|^2=\psi_r(x,y)^2+\psi_i(x,y)^2
\end{eqnarray*}
to multidimensional case. They demonstrated that
\begin{equation*}
\bLambda=(\Lambda(1),\ldots,\Lambda(n^2))^T=\psi_r\psi_r^T+\psi_i\psi_i^T,
\end{equation*}
where $\psi_r=(\psi_r(1),\ldots,\psi_r(n^2))^T$ and
$\psi_r=(\psi_i(1),\ldots,\psi_i(n^2))^T$, has non-central Wishart
distribution.

Since speckles are bigger than they appear in telescope, it is
highly misleading to assume the normality assumption, even if the
assumption of symmetry is taken, to study of planet formation. See
Figure 2 for the distribution of amplitude of wave in the focal
plane of a telescope.
\begin{figure}\label{amplitude}
\centering
\begin{tabular}{c}
\includegraphics[scale=.9]{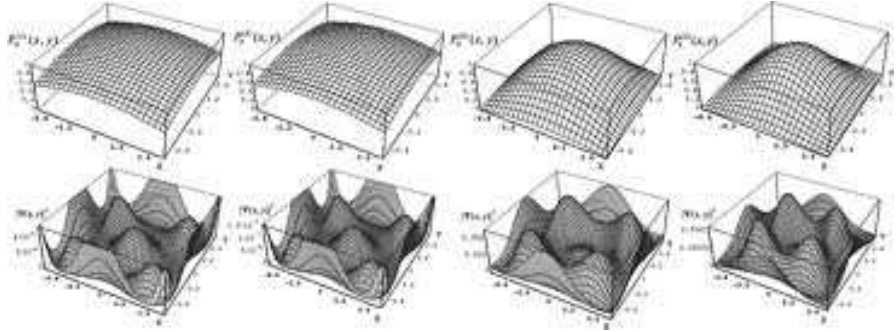}
\end{tabular}
\caption{Instantaneous intensity of the wave in the focal plane.}
\end{figure}

Thus it is more plausible to take these speckles as extremes in
amplitude of a wave or outlier in plane formation as appears in
telescope. In conclusion, accepting the assumption of symmetry,
the multivariate t-distribution (or may be lighter tail
alternative to Gaussian, as it might be captured from Figure 2) is
a relevant alternative to the normal one. Hence, by the theory
discussed in the above,
$\bLambda=(\Lambda(1),\ldots,\Lambda(n^2))^T$ has non-central GW
distribution arises from taking $g$ to be the kernel of
multivariate t-distribution in Eq. (2).

\section{\noindent Conclusion}
In this paper a family of distributions were introduced from the
Wishart generator distribution which includes the Wishart as a
special case. The Wishart generator distribution might be
important for a number of practical signal processing applications
including synthetic aperture radar (SAR), multi-antenna wireless
communications and direct imaging of extra-solar planets, the
latter was discussed using the non-central Wishart generator
distribution. Several statistical properties of this newly defined
distribution were studied from matrix theory viewpoint. Brief
notes regarding classical as well as Bayesian estimations were
also proposed.

\section{\noindent Appendix}

Initially let $\S_m$ and $\I_m$ be the spaces of all positive
definite matrices of order $m$ and all symmetric matrices between
0 and $\bI_m$ under the meaning of partial L\"owner ordering,
respectively. For a given matrix $\bA\in\mathbb{R}^{p\times p}$,
$\bA^T$ denotes the transpose of $\bA$,
$\tr(\bA)=a_{11}+\ldots+a_{pp}$; $\etr(\bA)=\exp(\tr(\bA))$;
$\det(\bA)=$ determinant of $\bA$; norm of $\bA=\|\bA\|=$ maximum
of absolute values of latent roots of the matrix $\bA$; and
$\bA^{\frac{1}{2}}$ denotes the unique square root of $\bA$.

Also denote the space of all orthogonal matrices of order $m$ by
\begin{eqnarray*}
\O(m)=\left\{\bH|\bH'\bH=\bI_m, \bH\bH'=\bI_m\right\}, \quad
\int_{\O(m)}\textnormal{d}\bH=1.
\end{eqnarray*}
Let $k$ be a positive integer; a partition $\kappa$ of $k$ is
written as $\kappa=(k_1,k_2,\ldots)$, where $\sum_rk_t=k$.
\begin{definition}(Muirhead, 2005)
Let $\bY$ be an $m\times m$ symmetric matrix with latent roots
$y_1,\cdots,y_m$ and let $\kappa=(k_1,\cdots,k_m)$ be a partition
of $k$ into not more than $m$ parts. The zonal polynomial of $\bY$
corresponding to $\kappa$, denoted by $C_\kappa(\bY)$, is a
symmetric, homogeneous polynomial of degree $k$ in the latent
roots $y_1,\cdots,y_m$ such that:
\begin{enumerate}
\item[(i)] The term of highest weight in $C_\kappa(\bY)$ is $y_1^{k_1},\cdots y_m^{k_m}$; that is,\\
(1) \quad $C_\kappa(\bY)=d_ky_1^{k_1},\cdots y_m^{k_m}$+terms of lower weight,\\
where $d_k$ is a constant.
\item[(ii)] $C_\kappa(\bY)$ is an eigenfunction of the differential operator $\triangle_{\bY}$ given by\\
(2) \quad $\triangle_{\bY}=\sum_{i=1}^m
y_i^2\frac{\partial^2}{\partial y_i^2}+\sum_{i=1}^m\sum_{j=1,j\neq
i}^m
\frac{y_i^2}{y_i-y_j}\;\frac{\partial}{\partial y_i}$.\\
\item[(iii)] As $\kappa$ varies over all partitions of $k$ the zonal polynomial have unit coefficients in the expansion of $(tr\bY)^k$; that is,\\
(3) \quad $(tr\bY)^k=(y_1+\cdots+y_m)^k=\sum_\kappa
C_\kappa(\bY)$.
\end{enumerate}\label{zonal}
\end{definition}
Immediate consequence of Definition \ref{zonal} is the following
important equality
\begin{eqnarray}\label{exponential}
\etr(\bX)=\sum_{k=0}^\infty
\frac{\tr(\bX)^k}{k!}=\sum_{k=0}^\infty\sum_\kappa
\frac{C_\kappa(\bX)}{k!}.
\end{eqnarray}
Let $\bX, \bY\in\mathbb{C}^{m\times m}$ then
\begin{eqnarray}\label{invariant}
C_\kappa(\bX)C_\tau(\bY)=\sum_{\phi\in\kappa\cdot\tau}\theta_\phi^{\kappa,\tau}C_\phi^{\kappa,\tau}(\bX,\bY)
\end{eqnarray}
where
$\theta_\phi^{\kappa,\tau}=\frac{C_\phi^{\kappa,\tau}(\bI_m,\bI_m)}{C_\phi(\bI_m)}$.

For any $\bX,\bY\in\S_m$, we have (Gross and Richards, 1987)
\begin{eqnarray*}
\int_{\O(m)}C_\kappa(\bX\bH\bY\bH')\textnormal{d}\bH=\frac{C_\kappa(\bX)C_\kappa(\bY)}{C_\kappa(\bI_m)}.
\end{eqnarray*}
For any $\bA\in\S_m$, we have (Gross and Richards, 1987)
\begin{eqnarray*}
\int_{0<\bX<\bI_m}\det(\bX)^{a-\frac{1}{2}(m+1)}C_\kappa(\bA\bX)\textnormal{d}\bX=\frac{\left(a\right)_\kappa
B_m\left(a,\frac{1}{2}(m+1)\right)}{\left(a+\frac{1}{2}(m+1)\right)_\kappa}C_\kappa(\bA).
\end{eqnarray*}
Let $a_1,\ldots,a_p$ and $b_1,\ldots,b_q$ be complex numbers, such
that for $1\leq i\leq p$ and $1\leq j\leq q$, $b_i>(j-1)/2$. Then
the hypergeometric function of one matrix argument is defined as
\begin{eqnarray*}
_pF_q(a_1,\ldots,a_p;b_1,\ldots,b_q;\bX)=\sum_{k=0}^\infty\sum
_\kappa\frac{(a_1)_\kappa,\ldots,(a_p)_\kappa}{(b_1)_\kappa,\ldots,(b_q)_\kappa}\frac{C_\kappa(\bX)}{k!},
\end{eqnarray*}
where $\sum_\kappa$  denotes the summation over all partition
$\kappa$, $\kappa=(k_1,\ldots,k_m)$, $k_1\geq k_2\geq\ldots\geq0$,
of $k$, and the generalized hypergeometric coefficient
$(b)_\kappa$ is given by
\begin{eqnarray*}
(b)_\kappa=\prod_{i=1}^m
\left(b-\frac{1}{2}(i-1)\right)_{k_i},\quad
(b)_k=b(b+1)\ldots(b+k-1),\quad (b)_0=1.
\end{eqnarray*}

The multivariate gamma function which is frequently used alongside
is defined as
\begin{eqnarray*}
\Gamma_m(a)=\int_{\S_m}\det(\bX)^{a-\frac{1}{2}(m+1)}\etr(-\bX)\textnormal{d}\bX=\pi^{\frac{1}{4}m(m-1)}\prod_{i=1}^m\Gamma\left(a-\frac{1}{2}(i-1)\right),
\end{eqnarray*}
where $\Re(a)>(m-1)/2$.

Multivariate beta function is defined as
\begin{eqnarray*}
B_m(a,b)=\int_{\I_m}\det(\bX)^{a-\frac{1}{2}(m+1)}\det(\bI_m-\bX)^{b-\frac{1}{2}(m+1)}\textnormal{d}\bX=\frac{\Gamma_m(a)\Gamma_m(b)}{\Gamma_m(a+b)},
\end{eqnarray*}
where $\Re(a),\Re(b)>(m-1)/2$.

\begin{definition}\label{definition Laplace}
The Laplace transform of the matrix valued function $f$ is given
by
\begin{equation}
g(\bY)=\L_f(\bX)=\int_{\S_m}\etr(-\bX\bY)f(\bX)\textnormal{d}\bX
\end{equation}
\end{definition}

\begin{definition}\label{Wishart} (Press, 1982)
A random matrix $\bV\in\S_m$ is said to have the non-singular
Wishart distribution with scale matrix $\bSigma\in\S_m$ and $n$
degrees of freedom, $m\leq n$, if the joint distribution of the
distinct elements of $\bV$ is continues with density
\begin{eqnarray*}
f(\bV)&=&c\det(\bSigma)^{-\frac{1}{2}n}\det(\bV)^{\frac{1}{2}(n-m-1)}\etr\left[-\frac{1}{2}\bSigma^{-1}\bV\right],
\end{eqnarray*}
where $c^{-1}=2^{\frac{mn}{2}}\Gamma_m\left(\frac{n}{2}\right)$.
It is denoted by $\bV\sim W_m(\bSigma,n)$.

Further if we take $\bU=\bV^{-1}$, then $\bU$ follows the inverted
Wishart distribution with scale matrix $\bSigma$ and $n$ degrees
of freedom denoted by $\bU\sim IW_m(\bSigma,n)$ with the following
density
\begin{eqnarray*}
g(\bU)&=&c\det(\bSigma)^{\frac{1}{2}n}\det(\bU)^{-\frac{1}{2}n-\frac{1}{2}(m+1)}\etr\left[-\frac{1}{2}\bSigma\bU^{-1}\right].
\end{eqnarray*}
\end{definition}

\begin{lemma}(Teng et al., 1989) Assume $\bZ$ is an $m\times m$
symmetric matrix, $\bX$ is an $m\times m$ complex symmetric matrix
with $\Re(\bX)\in\S_m$ and $h$ is a real function over
$\mathbb{R}^+$. Then
\begin{eqnarray*}
\int_{\S_m}\det(\bW)^{a-\frac{1}{2}(m+1)}C_\kappa(\bW\bZ)h(\tr\bX\bW)\textnormal{d}\bW=
\frac{(a)_\kappa\Gamma_m(a)\gamma_k(a)}{\Gamma(am+k)}\;\det(\bX)^{-a}C_\kappa(\bZ\bX^{-1}),
\end{eqnarray*}
where $\Re(a)>(m-1)/2$ and
\begin{eqnarray}
\gamma_k(a)=\int_{\mathbb{R}^+}y^{am+k-1}h(y)\textnormal{d}y.
\end{eqnarray}\label{lemma teng}
\end{lemma}

\section{\noindent Acknowledgements}
We would hereby acknowledge the support of the StatDisT group. We
would also want to thank Prof. Srivastava (University of Toronto)
for his help in Theorem 14. This work is based upon research
supported by the UP Vice-chancellor's post-doctoral fellowship
programme.

\section{\noindent References}

\baselineskip=12pt
\def\ref{\noindent\hangindent 25pt}

\ref A. Y. Abul-Magd, G. Akemann and P. Vivo, Superstatistical
generalizations of Wishart–Laguerre ensembles of random matrices,
{\em J. Phys. A: Math. Theo.}, 42 (2009), 175-207.

\ref S. Adhikari, Generalized Wishart distribution for
probabilistic structural dynamics, {\em Comp. Mech.}, 45 (2010),
495-511.

\ref C. Aime and R. Soummer, Influence of speckle and Poisson
noise on exoplanet detection with a coronograph, in {\em
EUSIPCO-04} (L. Torres, E. Masgrau, and M. A. Lagunas, eds.),
(Vienna, Austria), (2004), 509-512.

\ref T. W. Anderson, {\em An Introduction to Multivariate
Statistical Analysis}, 3rd Edition, John Wiley, New York, (2003).

\ref T. W. Anderson and K. T. Fang, On the theory of multivariate
elliptically contoured distributions and their applications, {\em
Technical Report}, No. 54, Department of Statistics, Stanford
University, California, (1979).

\ref M. Arashi, A. K. Md. Ehsanes Saleh, Daya K. Nagar, S. M. M.
Tabatabaey and H. Salarzadeh Jenatabadi, Bayesian Statistical
Inference For Laplacian Class of Matrix Variate Elliptically
Contoured Models, {\em Comm. Statist. Theo. Meth.}, Accepted
(2013).

\ref W. Bryc, Compound real Wishart and q-Wishart matrices, {\em
arXiv:0806.4014v1 [math.PR]} (2008).

\ref Jos\'e A. D\'iaz-Garc\'ia, Ram\'on Guti\'errez-J\'aimez, On
Wishart distribution: Some extensions, {\em J. Lin. Alg.} 435
(2011), 1296-1310.

\ref A. P. Dawid, and S. L. Lauritzen, Hyper-Markov laws in the
statistical analysis of decomposable graphical models. {\em Ann.
Statist.} 21 (1993), 1272-317.

\ref J. Fan, Non-central Cochran's theorem for elliptically
contoured distributions, {\em Acta Math. Sinica}, 3(2) (1986),
185-198.

\ref G. Frahm, {\em Generalized Elliptical Distributions: Theory
and Applications}, PhD Dissertation, Universit\"at zu K\"oln,
(2004).

\ref Francisco J. Caro-Lopera, Graciela Gonz\'alez-Far\'ias, N.
Balakrishnan, On generalized Wishart distributions - I: Likelihood
ratio test for homogeneity of covariance matrices, {\em Sankhya
A}, DOI 10.1007/s13171-013-0047-7, (2014).

\ref E. D. Feigelson and G. J. Babu, {\em Modern Statistical
Methods for Astronomy With R Applications}, Cambridge, UK, (2012).

\ref A. K. Gupta and D. K. Nagar, {\em Matrix variate
distributions} Chapman \& Hall/CRC, (2000).

\ref I. S. Gradshteyn and I. M. Ryzhik, {\em Table of Integrals,
Series, and Products}, 7th Ed., Academic Press, USA, (2007).

\ref K. I. Gross and D. ST. P. Richards, Special functions of
matrix argument I: Algebraic induction zonal polynomials and
hypergeometric functions, {\em Trans. Amer. Math. Soc.}, 301
(1987), 475–501.

\ref Anis Iranmanesh, M. Arashi, D. K. Nagar and S. M. M.
Tabatabaey, On Inverted Matrix Variate Gamma Distribution, {\em
Comm. Statist. Theo. Meth.}, 42(1) (2013), 28-41.

\ref G. Letac and H. Massam, The normal quasi-Wishart
distribution. In: Viana MAG, Richards DStP (eds) {\em Algebraic
Methods in Statistics and Probability}. AMS Contemp Math 287
(2001), 231–239.

\ref E. Lukacs, and R. G. Laha, Applications of Characteristic
Functions, {\em Griffin's Statistical Monographs \& Courses}, 14.
Charles Griffin \& Co., Ltd., London, (1964).

\ref Rob J. Muirhead, {\em Aspect of Multivariate Statistical
Theory}, 2nd Ed., John Wiley, New York, (2005).

\ref S. Munilla and R. J. C. Cantet,  Bayesian conjugate analysis
using a generalized inverted Wishart distribution accounts for
differential uncertainty among the genetic parameters – an
application to the maternal animal model, {J. Anim. Breed.
Genet.}, 129 (2012), 173–187.

\ref S. J. Press, {\em Applied Multivariate Analysis}, Holt,
Rinchart \& Winston, New York, (1982).

\ref A. Roverato, Hyper-Inverse Wishart Distribution for Non-
decomposable Graphs and its Application to Bayesian Inference for
Gaussian Graphical Models, {\em Scandinavian J. Statist.}, 29
(2002), 391-411.

\ref M. S. Srivastava and C. G. Khatri, {\em An Introduction to
Multivariate Analysis}, North-Holland, Amsterdam, (1979).

\ref B. C. Sutradhar, and M. M. Ali, A generalization of the
Wishart distribution for the elliptical model and its moments for
the multivariate t model, {\em J. Mult. Anal.}, 29 (1989),
155-162.

\ref C. Teng, H. Fang, W. Deng, The generalized noncentral Wishart
distribution, {\em L. Math. Res. Exp.}, 9(4) (1989), 479-488.

\ref J. Y. Tourneret, A. Ferrari, G. Letac, The noncentral Wishart
distribution: Properties and application to speckle imaging, in:
{\em Proc. IEEE Workshop on Stat. Signal Proc.}, (2005), 1856–1860

\ref H. Wang and M. West, Bayesian analysis of matrix normal
graphical models, {\em Biometrika}, 96(4) (2009), 821-834.

\ref C. S. Withers and S. Nadarajah, log det A = tr log A, {\em
Int. J. Math. Education Sci. Tech.}, 41(8) (2010), 1121-1124

\ref C. S. Wong, T. Wang, Laplace-Wishart distributions and
Cochran theorems, {\em Sankhya} 57 (1995), 342-359.

\end{document}